\documentclass[a4paper]{amsart}

\usepackage{amsmath,amssymb,amsthm}
\usepackage{fullpage}
\usepackage{hyperref}
\usepackage{graphicx}
\usepackage{stackrel}
\usepackage{tikz} 
\usetikzlibrary{matrix,arrows}
\usepackage{color}
\usepackage{textcomp}

\newtheoremstyle{myremark} % name
    {7pt}                    % Space above
    {7pt}                    % Space below
    {}  	                 % Body font
    {}                           % Indent amount
    {\bf}       	         % Theorem head font
    {.}                          % Punctuation after theorem head
    {.5em}                       % Space after theorem head
    {}  % Theorem head spec (can be left empty, meaning ‘normal’)
    
\theoremstyle{plain}
\newtheorem{lemma}{Lemma}[section]

\newtheorem{corollary}[lemma]{Corollary}
\newtheorem{proposition}[lemma]{Proposition}
\newtheorem{theorem}[lemma]{Theorem}
\newtheorem*{theorem-main}{Theorem~\ref{thm:main}}
\newtheorem*{theorem-secondary}{Theorem~\ref{thm:secondary}}
\theoremstyle{definition}

\newtheorem{definition}[lemma]{Definition}

\theoremstyle{myremark}
\newtheorem{remark}[lemma]{Remark}

\newtheorem{example}[lemma]{Example}

\newcommand{\p}{{\mathfrak{p}}}
\newcommand{\m}{{\mathfrak{m}}}
\newcommand{\q}{{\mathfrak{q}}}
\newcommand{\I}{{\mathcal{I}}}

\newcommand{\sheight}{{^*\mathrm{ht}}}
\newcommand{\height}{{\mathrm{ht}}}
\newcommand{\sdim}{{\phantom{}^*\mathrm{dim}}}

\newcommand{\xra}{\xrightarrow}

\newcommand{\grmod}{\mathfrak{grmod}}

\newcommand{\vdim}{\mathrm{vdim}}

\begin{document}

\title{Length and Multiplicities in Graded Commutative Algebra}
\author{Mark Blumstein}
\date{\today}

\maketitle

\setcounter{tocdepth}{1}
\tableofcontents
\section{Introduction}

This paper is a synthesis of the main ideas from the author's matster's thesis. The author would like to thank Jeanne Duflot for her steady guidance and dedication as advisor. This paper came to be via the study of the commutative algebra of equivariant cohomology rings $H^*_G(X)$ associated to a group $G$ acting on a topological space $X$, which are of course naturally graded. This study was really begun about 50 years ago by Quillen \cite{Quillen1}, \cite{Quillen2} who described the (Krull) dimension of these graded rings and gave a decomposition of their spectra (in the sense of algebraic geometry).  Many authors have followed with their own studies of these rings from the point of view of commutative algebra; recent contributions include the work of Symonds \cite{Symonds}, Lynn \cite{Lynn} and Duflot \cite {Duflot3}.

We were attempting to generalize the work of Lynn \cite{Lynn}, resulting in the paper \cite{BD}, and found that we needed a careful exploration of various notions of multiplicity for graded rings which were nonstandard in two ways:   not  positively graded (for example, they might be graded localizations) and/or not generated by elements of degree 1.  Since, as algebraic topologists, the degree of homogeneous elements in these rings can have geometric or representation-theoretic meaning to us, we were not comfortable with using the geometer's trick of the Veronese embedding to get around the second problem.

In this paper, there is nary a word about equivariant cohomology or algebraic topology.  It is all about graded commutative algebra and much of it is expository.  The main results of interest to us are the theorems about length, multiplicity and a second notion of ``degree" (which is really another sort of multiplicity) for the Poincar\`{e} series of a graded ring.  When we were able to write results for rings from a larger collection than simply those of cohomology type, we tried to do so. We hope that workers in fields other than algebraic topology find this exposition useful.

\section{A Review of Standard Definitions and Facts in the Graded Category}
\label{Introductory Results in the Graded Context}
We consider only strictly commutative $\mathbb{Z}$-graded rings and modules in this paper and use the standard notation:  if $A$ is a graded ring, $M$ is a graded $A$-module and $n \in \mathbb{Z}$, $M_n$ is the set of homogeneous elements of degree $n$ (although ``degree" will also have another meaning here);  for every $x \in M$, $x$ may be written uniquely as $x= \Sigma_{n \in \mathbb{Z}}x_n$, where $x_n \in M_n$ and $x_j = 0$ for all but finitely many $j$, the element $x_n$ is   the homogeneous component of $x$ of degree $n$.  An element $x \in M$ is a homogeneous element if and only if $x$ has at most one nonzero homogeneous component.  If $d \in \mathbb{Z}$, we use the following convention for the suspended $A$-module $M(d)$:  $$M(d)_j \doteq M_{d+j},$$ for every $j \in \mathbb{Z}$; also, if $M$ and $N$ are graded $A$-modules, and $\psi: M \to N$ is an $A$-module homomorphism, then $\psi$ is a \textit{graded homomorphism of degree $d$} if for every integer $n$, $\psi(M_n) \subseteq N_{n+d}$.

\begin{definition} Suppose $A$ is a graded ring. The category $\grmod (A)$ has objects finitely generated graded $A$-modules. The morphisms of $\grmod(A)$ are the $A$-module homomorphisms which are graded of degree zero (i.e. degree-preserving).

\end{definition}

Recall that a submodule $N$ of a graded $A$-module $M$ is a graded submodule if and only if it is generated over $A$ by homogeneous elements; this is equivalent to the condition that for every element of $N$, all of its homogeneous components are in $N$.

Whether or not $M$ and $A$ are graded, the set of associated primes for $M$ in $A$ is denoted by $Ass_A(M)$ and the support of an $A$-module $M$ is the set $Supp_A(M) \doteq \{ \p \in Spec(A) : M_\p \neq 0 \}$.  For $M$ finitely generated over $A$, $\p \in Supp_A(M)$ if and only if $Ann_A(M) \subseteq \p$. A prime ideal of $A$ that contains $Ann_A(M)$, and is minimal amongst all primes containing $Ann_A(M)$ is called a \textit{minimal prime} for $M$.  If $M= A/\I$ for an ideal $\I$ of $A$, then a minimal prime for $M$ is called a minimal prime in $A$ over $\I$.  Note that $Ass_A(M) \subseteq Supp_A(M)$.

\begin{definition}The graded support of $M$, $*Supp_A(M)$, is the set of all graded prime ideals in the support of $M$.  If $\I$ is a graded ideal in $A$, the graded variety of $\I$, $*V(\I)$, is the set of all graded primes in $A$ containing $\I$.  (Recall that if $\mathcal{J}$ is any ideal in $A$, graded or not, $V(\mathcal{J})$ is the set of prime ideals in $A$ containing $\mathcal{J}$.)  \end{definition}

We collect some standard results about $Ann_A(M)$ and $Ass_A(M)$ for the graded category below.

\begin{proposition}\label{lemma Ass is graded}
Let $A$ be a graded ring with $M$ a graded $A$-module.
\begin{itemize}
\item[i)]  $Ann_A(M)$ is a graded ideal in $A$ and $Ann_A(M) = Ann_A(M(d))$ for every $d \in \mathbb{Z}$.
\item[ii)] If $\p \in Ass_A(M)$, then $\p$ is a graded ideal of $A$ and  is the annihilator of a homogeneous element in $A$.
\item[iii)]  Therefore, if $\I$ is a graded ideal in $A$, all primes in $Ass_A(A/\I)$ are graded.
\item[iv)]  If $\p$ is a minimal prime for $M$, then $\p \in Ass_A(M)$; thus, all minimal primes for $M$ are graded.
\end{itemize}
\label{lemma minimal prime}
\end{proposition}

Finally, for an ideal $\I$ in a graded ring,  graded or not,  $\I^*$ is defined as the largest, graded ideal contained in $\I$; i.e. $\I^*$ is the ideal generated by all homogeneous elements of $\I$; if $\p$ is a prime ideal in $A$, $\p^*$ is also a prime ideal in $A$.  

\subsection{Noetherian graded rings}
When we say that  a graded $A$-module $M$ is a Noetherian $A$-module, we mean that it is Noetherian in the usual sense, forgetting the grading. 

One can show \cite{brhe} that the following conditions on $A$ are equivalent:
\begin{itemize}
\item $A$ is Noetherian.
\item Every graded ideal in $A$ is  generated by a finite set of homogeneous elements.
\item $A_0$ is Noetherian and $A$ is a finitely generated $A_0$-algebra by a set of homogeneous elements.
\end{itemize}

So, if $M$ is a finitely generated graded $A$-module, and $A$ is Noetherian, then $M$ is Noetherian, and
\begin{itemize}
\item every $A$-submodule $N$ of $M$ is finitely generated over $A$, and if $N$ is graded, it is generated over $A$ by a finite set of homogeneous elements;
\item for every $j$, $M_j$ is a Noetherian $A_0$-module and so every $A_0$-submodule of $M_j$ is finitely generated:  if one has an ascending chain $X_1 \subseteq X_2 \subseteq \cdots $ of $A_0$-submodules of $M_j$, then letting $AX_i$ be the (graded) $A$-submodule generated by $X_i$, we must have $AX_i = AX_{i+1}$ for all $i$ greater than or equal to  some fixed $N$.  By degree considerations, $X_i =  AX_i \cap M_j$ for every $i$, so $X_i = X_{i+1}$ for $i \geq N$.
\end{itemize}

\begin{lemma} If $A$ is a graded Noetherian ring and $M \in \grmod(A)$, 
\begin{itemize}
\item[a.] $ Supp_A(M) = V(Ann_A(M)) \doteq V(M)$, so that $*Supp_A(M) = *V(Ann_A(M)) \doteq *V(M).$
\item[b.]  If $\I$ is a graded ideal in $A$, $*V(M/\I M) = *V(M) \cap *V(\I) = *V(Ann_A(M) + \I)$.
\item[c.]  If $\I$ and $\tilde{\I}$ are two graded ideals in $A$ then their radicals are also graded, and  $*V(\I) = *V(\tilde{\I})$ if and only if $\sqrt{\I} = \sqrt{\tilde{\I}}$.
\end{itemize}
\end{lemma}
\begin{proof}  The proof of $a.$ can be found in \cite{se}; also \cite{se} tells us that $V(M/\I M) = V(M) \cap V(\I)= V(Ann_A(M) + \I)$ and so $b.$ follows from this.  For c., the forward implication follows since all minimal primes over $\I$ are graded, thus occur as minimal elements both in $V(\I)$ and $*V(\I)$, and $\sqrt{\I}$ is the intersection of the (finite number of) minimal primes over $\I$.  
\end{proof}

The type of filtration described in the following lemma will be used several times in this paper; and we provide a brief discussion of its proof.

\begin{lemma}
\label{corollary the filtration graded}
If $A$ is a Noetherian graded ring and  $M \in \grmod(A)$ is nonzero, there exists a finite filtration $M^{\bullet}$ of $M$ by graded submodules ($M^0 = 0; M^{n} = M$), integers $d_j$ and graded primes $\p_j \in Spec(A)$ with  graded isomorphisms of graded $A$-modules, $M^{i+1}/M^{i} \cong A / \p_{i+1} (-d_{i+1})$.  Furthermore, given a finite list of  graded primes $(\p_j \mid 1 \leq j \leq n )$ in $Spec(A)$ (not necessarily distinct), and a graded filtration $M^{\bullet}$ of $M$ by graded submodules as above, we must have $$Ass_A(M) \subseteq \{\p_j \mid 1 \leq j \leq n \} \subseteq *Supp_A(M)$$ and these three sets must have the same minimal elements, the set of which consists of the minimal primes of $M$.   Finally, if $\p$ is a minimal prime for $M$, forgetting all gradings and using the fact that  the ordinary localization $M_{\p}$ is a finitely generated Artinian $A_{\p}$-module, the number of times that $A/\p$, possibly shifted, occurs as a graded $A$-module isomorphic to a subquotient of $M^{\bullet}$ is always equal to the length of $M_{\p}$ as an $A_{\p}$-module and is thus independent of the choice of the graded filtration $M^{\bullet}$.
\end{lemma}

\begin{proof} We remind the reader of the proof of the first statement, adapted to the graded case:
Using the Noetherian hypothesis, since $M \neq 0$, $Ass_A(M) \neq \emptyset$, so we may pick an element $\p_1 \in Ass_A(M)$.  Then $\p_1$ is graded and there exists a homogeneous element $m_ 1\in M$ such that $\p_1 = ann_A(m_1)$. Suppose $\deg(m_1) = d_1$, then $A/\p_1(-d_1)$ is graded isomorphic to a graded $A$-submodule of $M$ which we call $M^1$.  

If $M^1 = M$, we are done.  If not, we take the $A$-module $M/M^1$, notice that it is nonzero, and produce an associated prime $\p_2 \in Ass_A(M/M^1)$.  Since $M/M^1$ is a graded $A$-module  $\p_2$ is also graded.  Suppose $\p_2 = Ann_A(\overline{m}_2)$ where $m_2 \notin M^1$ is a homogeneous element in $M$ and $\deg(m_2) = d_2$; $\overline{m}_2$ is the coset of $m_2$ in $M/M^1$.  Thus there is a graded submodule $M^1 \subseteq M^2$ such that $M^2/M^1$ is graded isomorphic to $A/\p_2(-d_2)$.  At some point there must be a smallest $n\geq 1$ with $M^n = M$, since otherwise the Noetherian hypothesis would be violated. 

For the last two statements, we refer to \cite{se}.  
\end{proof}

\section{Graded Length}

We've already started using the notation ``$^*P$" for a modification of  a property or definition $P$ in the ungraded category to obtain a property or definition in the graded category, and we continue it in this section.  From now on, unless stated otherwise, all modules and rings are graded, although we will sometimes redundantly restate this.
\begin{definition}  If $A$ is a graded ring, a graded ideal $\mathcal{N}$ is *maximal if and only if $\mathcal{N} \neq A$ and $\mathcal{N}$ is a maximal element in the set of all proper graded ideals of $A$.  
\end{definition}

\begin{definition}
A *simple $A$-module is a nonzero graded $A$-module with no nonzero proper graded submodules.  A *composition series for a graded module $M \in \grmod(A)$ is a chain of graded $A$-submodules of $M$, $0=M^0 \subset \cdots \subset M^n = M$ such that each successive quotient $M^i/ M^{i-1}$ is isomorphic as a graded $A$-module to a *simple module. The *length of the *composition series $0=M^0 \subset \cdots \subset M^n = M$ is defined to be $n$.
\end{definition}

The fundamental theorem about *composition series mirrors that in the ungraded case.  The proof of the following is nearly identical to the ungraded case (\cite{Eis},Theorem 2.13), with only minor adjustments made to account for the grading, and we leave this effort to the reader.  
\begin{theorem} 
Suppose for $M \in \grmod(A)$ that a *composition series of length $n$ for $M$ exists.  Then, every chain of graded submodules of $M$ has length $\leq n$, and can be refined to a *composition series of length $n$.  Every *composition series for $M$ has length $n$.  \end{theorem}

\begin{definition}  If $M$ has a *composition series as an $A$-module,
the *length of $M \in \grmod(A)$ is defined to be the length of a *composition series for $M$.  We use the notation $*\ell_A(M)$ for this number; as usual, we say $*\ell_A(M) = \infty$ if $M$ does not have a *composition series.
\end{definition}

If we forget all gradings on $A$ and $M$, $\ell_A(M)$ denotes the usual length of $M$ as an $A$-module.

Some properties of $*\ell_A$ are as expected: 
\begin{itemize}
\item  If $0 \rightarrow M \rightarrow N \rightarrow P \rightarrow 0$ is an exact sequence in $\grmod(A)$, then $N$ has a *composition series if and only if both $M$ and $P$ do; and in this case, $*\ell_A(N) = *\ell_A(M) + *\ell_A(P)$.
\item If $d \in \mathbb{Z}$, then $*\ell_A(M(d)) = *\ell_A(M)$.
\end{itemize}

The only simple $A$-modules in the ungraded case are $A$-modules of the form $A/\m$, where $\m$ is a maximal ideal of $A$ (recall all rings are commutative).  Thus we are led to define graded fields; these are the rings of *length zero as modules over themselves.

\begin{theorem} \cite{Eis}
\label{theorem gr field}
Let $F$ be a graded ring.  The following are equivalent:
\begin{enumerate}
\item Every nonzero homogeneous element in $F$ is invertible.
\item $F_0$ is a field and either $F = F_0$, or there exists a $d>0$ and an $x \in F_d$ such that  $F \cong F_0[x,x^{-1}]$ as a graded ring.  In fact, in this last case,  $d>0$ is the smallest positive degree with $F_d \neq 0$.  
\item The only graded ideals in $F$ are $F$ and $0$. 
\end{enumerate}

\noindent A ring satisfying any of these three equivalent conditions is called a \textbf{graded field}.
\end{theorem}

\begin{example} \label{example graded field}  If $F$ is a graded field with a nonzero positive degree element, then $F$ is *simple as a module over itself, but it is not simple as such.  To see this write $F = F_0[t,t^{-1}]$, with $\deg(t) = d >0$, and $F_0$ a field. So $F$ is certainly *simple, but if $\mathcal{J}$ is the ungraded ideal generated by $t+1$, $\mathcal{J}$ is a nonzero proper $F$-submodule of $F$, so $F$ is not simple.  Furthermore, $F$ has a unique *maximal ideal, the zero ideal, but has as least as many ungraded nonzero maximal ideals as the nonzero elements of $F$.  While $F$ has a *composition series, it has no composition series.

\end{example}

Similarly to the ungraded case, $M$ is a *simple $A$-module if and only if there exists a *maximal ideal $\mathcal{N}$ of $A$, an integer $d$ and a graded $A$-module isomorphism $M \cong (A/\mathcal{N})(d)$:  if $M$ is *simple, let $x$ be any nonzero homogeneous element of $M$, say $\deg(x) = -d$.  Then, the submodule of $M$ generated by $x$ is nonzero and graded, so must be all of $M$.  The homomorphism $A(d) \rightarrow M$ of graded $A$-modules defined by $a \mapsto ax$ is thus surjective; its kernel is a graded ideal in $A(d)$ of the form $\mathcal{N}(d)$ for some graded ideal $\mathcal{N}$ of $A$; since $M$ is *simple, $\mathcal{N}$ must be *maximal.  The converse is left to the reader. 

Other facts parallel to the ungraded case include:   1) for every proper  graded ideal $\mathcal{I}$ in $A$, there exists a *maximal ideal $\mathcal{N}$ containing $\mathcal{I}$;  2) if $\mathcal{N}$ is a proper graded ideal of $A$, then $\mathcal{N}$ is *maximal if and only if  $A/\mathcal{N}$ is a graded field.  Thus, every *maximal ideal in $A$ is a graded prime ideal. Furthermore, if $\mathcal{N}$ is *maximal in $A$, then $\mathcal{N}_0$ is a maximal ideal in $A_0$.

The structure of finitely generated graded modules over graded fields mirrors that for the ungraded category:

\begin{lemma}  Suppose that $M$ is a finitely generated graded module over a graded field $F = F_0[t,t^{-1}]$, where $t$ has positive degree $d$ and  $F_0$ is a field.  Then
\begin{enumerate}
\item[a)] $M$ is a free graded  $F$-module, of finite rank, on a set of homogeneous generators.
\item[b)]  $M_0$ is a finite-dimensional vector space over $F_0$ of $F_0$-dimension less than or equal to the rank of $M$ over $F$.
\end{enumerate}
\end{lemma}
\begin{proof}  Assume $M \neq 0$.  Say $M$ is finitely generated over $F$ by homogeneous elements $e_1, \ldots, e_r$, where $r\geq 1$ is the minimal number for a homogeneous generating set for $M$ as an $F$-module.  Then, $M$ is free on the $e_js$:  certainly this set spans $M$ over $F$.  Suppose that there is a relation 
$\sum_j \alpha_je_j = 0,$ with $\alpha_j \in F$.  We may assume that all the $\alpha_js$ are homogeneous.  If $\alpha_r \neq 0$, then it is invertible in $F$, so
$\sum_{j=1}^{r-1} \alpha_r^{-1}\alpha_je_j +e_r= 0$, implying that $r$ is not minimal.  Therefore $\alpha_r = 0$; and continuing the process, $\alpha_j=0$ for every $j$.

Set $d_j = \deg e_j$.  Now, note that $X \doteq \{t^{-d_j/d}e_j \mid 1 \leq j \leq r  \;\mbox{and}\; d \; \mbox{divides} \; d_j \}$ is a basis for $M_0$ over $F_0$; of course, if $d$ does not divide any $d_j$, then  $M_0 = 0$.  To see this, note that $X$ is linearly independent over $F_0$, since the $e_js$ are linearly independent over $F$.  If $x \in M_0$, then $x = \sum_j \alpha_je_j$, where $\alpha_j$ is a homogeneous element of $F$ and $\deg \alpha_j + d_j = 0, \forall j$.  Now, if $\alpha_j \neq 0$, $d$ divides its degree, by definition of $F$.   Thus, $d$ divides $d_j$ for every $j$ such that $\alpha_j \neq 0$.  If $d$ divides $d_j$, then $\alpha_j = \beta_j t^{-d_j/d},$ where $\beta_j \in F_0$.   Thus $x$ is in the $F_0$-span of $X$.  
\end{proof}

\begin{definition}
$M \in \grmod(A)$ is said to be a *Artinian module if $M$ satisfies DCC on all chains of graded $A$-submodules of $M$.
\end{definition}

Unlike the Noetherian case, an $A$-module $M$ can be *Artinian without being Artinian:  an example is given by $A=M$, where $A$ is a graded field with a nonzero positive degree element.  

Similarly to the ungraded case, we have
\begin{lemma}  Suppose that $A$ is a graded Noetherian ring and $M \in \grmod(A)$.  Then the following are equivalent:
\begin{itemize}
\item[a)] $M$ is *Artinian.
\item[b)] $*\ell_A(M) < \infty.$
\item[c)] $*V(M)$ consists of a finite number of *maximal ideals.
\end{itemize}
\end{lemma}
\begin{proof}  The proof of the equivalence of a) and b)  in the ungraded case, as in \cite{AtMac}, adapts in a straightforward way to the graded case. Note that the proof of ``b) implies a)" does not require $A$ to be Noetherian.

 To see how b) implies c), assume that $M$ has a *composition series
$$0=M^0 \subset M^1 \subset \cdots \subset M^{n-1} \subset M^n=M;$$ the *simplicity of the subquotients means that there are *maximal graded ideals $\m_i$ of $A$ and integers $d_i$ such that $M^i/M^{i-1} \cong (A/\m_i)(d_i)$ as graded $A$-modules.  Thus, $\m_1 \m_2 \cdots \m_n \subseteq Ann_A(M)$.  If $\p$ is a prime minimal over $Ann_A(M)$, then we have seen that $\p$ is graded.  Since $\m_1 \cdots \m_n \subseteq \p$, we must have $\m_i \subseteq \p$ for at least one $i$.  But $\m_i$ is *maximal, so $\m_i =\p$.  Therefore $*V(M) \subseteq \{\m_1, \ldots, \m_n \}$.  

For c) implies b), since  $\sqrt{Ann_A(M)}$ is the intersection of the primes minimal over $Ann_A(M)$, and there are a finite number of  these,  all graded, the hypothesis implies that this finite list of primes consists entirely of *maximal ideals; say these ideals are $\m_1, \ldots, \m_n$.  Thus, there is an $N$ such that $(\m_1 \cdots \m_n)^N \subseteq Ann_A(M)$ and  there is a sequence $\tilde{\m}_1, \ldots, \tilde{\m}_{nN}$ of *maximal ideals in $A$, not necessarily distinct, whose product is contained in $Ann_A(M)$.  Analogously to  the ungraded case, one can then construct a *composition series for $M$.
\end{proof}

If $V$ is a graded vector space over an ordinary field $k$, where $k$ is  regarded as a graded ring concentrated in degree zero (this means that all nonzero elements have degree zero), then $$\vdim_k(V) \doteq \; \mbox{the dimension of } \; V \; \mbox{as a vector space over} \; k.$$

\begin{lemma} \label{lemma length of tensor}
Let $A$ be a graded Noetherian ring which is a finitely generated graded algebra over a field $k \subseteq A_0$, $M \in \grmod(A)$, $V$ a graded finite dimensional vector space over $k$, and say that $\vdim_k(V) = d$. If $a \in A$, $m \otimes v \in M \otimes_k V$, then give $M \otimes_k V$ an $A$-module structure by $a \cdot (m \otimes v) \doteq (a \cdot m) \otimes v$, and grade $M \otimes_k V$ in the usual way. Then $$*\ell_A(M\otimes_k V) = *\ell_A(M)\cdot d.$$  
\end{lemma}
\begin{proof}  Since $V$ is finite dimensional over $k$, we may suspend $V$ appropriately and assume, without loss of generality, that $V_j = 0$ for $j <0$; in this case, there exists an $n$ such that $j > n$ implies $V_j =0$. Define a graded filtration of $M\otimes_k V$ by graded $A$-modules: $\mathcal{F}^i \doteq M \otimes_k (V_0 \oplus \cdots \oplus V_{n-i})$ for $0 \leq i \leq n$, and $\mathcal{F}^{n+1} \doteq 0$. Consider that $\mathcal{F}^i/\mathcal{F}^{i+1} \cong M \otimes_k V_{n-i}$, and the additive property of length allows $*\ell_A(M\otimes_k V) = \sum_{i=0}^n *\ell_A(\mathcal{F}^i/\mathcal{F}^{i+1}) = \sum_{i=0}^n*\ell_A(M \otimes_k V_{n-i})$.

By hypothesis, each graded component $V_j$ of $V$ is a finite dimensional graded vector space concentrated in degree $j$.  Thus, there is a graded isomorphism for each $j$, $V_j \cong k^{f(j)}(-j)$ where $f(j)$ is a function giving the vector space dimension of $V_j$.  Since $M \otimes_k V_j \cong M \otimes_k k^{f(j)} \cong \oplus_1^{f(j)} M(-j)$, we have that $*\ell_A(M \otimes_A V_j) = *\ell_A(M)\cdot f(j)$.  By hypothesis, $\sum_{j=0}^n f(j) = d$, the total vector space dimension of $V$, and finally $*\ell_A(M \otimes_k V) = \sum_{i=0}^n*\ell_A(M \otimes_k V_{n-i}) = \sum_{i=0}^n *\ell_A(M)\cdot f(n-i) = *\ell_A(M) \sum_{i=0}^n f(n-i) = *\ell_A(M) \cdot d$.
\end{proof}

\subsubsection{Positively or Negatively Graded Rings}

 A graded ring $S$ is positively (resp. negatively) graded if and only if $S_i = 0$ for $i <0$ (resp. $i>0$).  The graded ideal $S_+$  (resp. $S_-$) of $S$ is defined as $\oplus_{i >0} S_i$ (resp. $\oplus_{i<0} S_i$).  Note that if $M \in \grmod(S)$, since $S$ is positively (resp. negatively) graded,  there exists an integer $e$ such that $M_i =0$ for all $i <e$ (resp. $i >e$).  Also, for a proper, graded ideal $\m$ of $S$, the following are equivalent:
\begin{itemize}
\item $\m$ is *maximal in $S$. 
\item $\m = \m_0 \oplus S_+,$ (resp. $\m_0 \oplus S_-$) and $\m_0$ (the degree zero elements of $\m$) is a maximal ideal in $S_0$.
\item  $S/\m$ is a graded field, concentrated in degree zero; i.e. $S/\m$ is an ordinary field.
\item $\m$ is a maximal ideal in $S$.
\end{itemize}

For positively or negatively graded rings, there is  no difference between *length and length:

\begin{lemma} \label{lemma pos graded length equals *length} Suppose that $S$ is a positively or negatively graded Noetherian ring, and $M \in \grmod(S)$ is such that $*\ell_S(M) < \infty$.  Then, $*\ell_S(M) = \ell_S(M)$.  
\end{lemma}

\begin{proof}  Since $*V(M)$ consists of a finite number of *maximal ideals, there is a sequence of graded $S$-modules
$$0=M^0 \subset M^1 \subset \cdots \subset M^{n-1} \subset M^n=M,$$  *maximal graded ideals $\m_i$ of $S$ and integers $d_i$ such that $M^i/M^{i-1} \cong (S/\m_i)(d_i)$ as graded $S$-modules.  By the remark above,   $S/\m_i$ is concentrated in degree 0 and each $\m_i$ is a maximal ideal in $S$. So,  forgetting gradings everywhere, the given *composition series is a composition series.
\end{proof}

Even in the cases where *length and length coincide, we'll usually just talk about *length, emphasizing constructions using graded modules only.  For example,

\begin{lemma}  \label{lemma *length of components} Suppose $S$ is a positively graded ring and  $X \in \grmod(S)$.
\begin{itemize}
\item[a)] If $*\ell_S(X) < \infty$, there exists an integer $J$ such that if $j > J$, then $X_j= 0$.  
\item[b)] If  $S_i$ is finitely generated as an $S_0$-module for every $i$, then $X_j$ is a finitely generated $S_0$-module, for every $j$.
\item[c)] Suppose  $S_0$ is Artinian, $S_i$ is finitely generated as an $S_0$-module for every $i$, and there exists an integer $J$ such that if $j >J$, then $X_j = 0$. Then,  $\ell_{S_0}(X_j) < \infty$ for every $j$,  and  $*\ell_S(X) = \ell_{S_0}(X) < \infty$, where $\ell_{S_0}(X) \doteq \sum_j \ell_{S_0}(X_j)$ is the (total) $S_0$-length  of $X$.
\end{itemize}
\end{lemma}
\begin{proof} For every $t \in \mathbb{Z}$ define $X_{\geq t} \doteq \oplus_{s \geq t} X_s$.  Since $S$ is positively graded, $X_{\geq t}$ is a graded $S$-submodule of $X$.  Since $X$ is finitely generated over $S$, and $S$ is positively graded,  there exists a $t_0 \in \mathbb{Z}$ such that $X_{\geq t_0} = X$.  So we have a descending chain of graded $S$-submodules of $X$
$$ \cdots \subseteq X_{\geq t_0 +k} \subseteq X_{\geq t_0 + k-1} \subseteq \cdots  \subseteq X_{\geq t_0 + 1} \subseteq X_{\geq t_0} = X.  (*)$$ 

For a), if   $*\ell_S(X) < \infty, $ $X$ is *Artinian, so this chain stabilizes.  By definition, this means that there exists an $J \geq t_0$ such that $X_j = 0$ for $j >J$.

 For b), let $t_0$ be defined as in the first  paragraph above; assume that $X_{t_0} \neq 0$.  Then, one can prove, by induction on $j$, that each $X_j$ is finitely generated over $S_0$ as follows.  If $j = t_0$, then since $X$ is finitely generated as an $S$-module, say by $x_1, \ldots x_N$, if $\beta_{t_0} = \{ x_i \mid \deg(x_i) = t_0\}$, $X_{t_0}$ must be generated by $\beta_{t_0}$ as an $S_0$-module.  Assume that $j > t_0$ and $X_u$ is finitely generated over $S_0$ for $u <j$.  Then,   $X_{<j} = \oplus_{u=t_0}^{j-1} X_u = \oplus_{u<j}X_j$, is finitely generated over $S_0$. Choose a finite set $\beta_{<j}$ of homogeneous elements that generate $X_{<j}$ over $S_0$. Choose finite generating sets $\alpha_u$ for each $S_u$ over $S_0$.  Let $\beta_j = \{ x_i \mid \deg(x_i) = j \}$.  The claim is that the finite set $B_j \doteq \{ a e \mid a \in \alpha_u, e \in \beta_{<j} \; \mbox{and} \; u+ \deg(e) = j \} \cup \beta_j$ spans $ X_j$ over $S_0$:  if $x \in X_j$, then $x = \sum_i a_i x_i$, with $a_i$ homogeneous in $S$ for every $i$, and  if $a_ix_i \neq 0$, $\deg(a_i)  + \deg(x_i) = j$; from now on we'll just talk about the indices $i$ such that $a_ix_i \neq 0$.  If $\deg(a_i) = 0$, then $x_i \in \beta_j \subseteq B_j$ and $a_i \in S_0$.  If $\deg(a_i) >0$, then $\deg(x_i)$ is strictly less than $j$ so that $x_i$ is in the $S_0$-span of $\beta_{<j}$; certainly,  $a_i$ is in the $S_0$-span of $\alpha_{\deg(a_i)}$, so   $a_ix_i$ is in the $S_0$-span of $ B_j$. 

For c),  given $J$ such that $X_J = 0$ for $j >J$, and choosing $t_0 \leq J$ such that $X_j = 0$ for $j < t_0$, the chain (*) terminates at the left in 0, and has successive quotients isomorphic to a graded $S$-module $X_j$ (concentrated in degree $j$), where the $S$-module structure is determined by $rx = 0$ if $r \in S_+$.  Since b) says that each $X_j$ is finitely generated over $S_0$,  and $S_0$ is Artinian, the chain (*) may be refined to a *composition series of $X$, of length equal to $\sum_j \ell_{S_0}(X_j)$.

\end{proof}

\section{Graded Localization}

Localizing in the graded category can be done in a few ways.  We may localize as usual, forgetting the graded structures, we may localize at sets consisting of homogeneous elements, or as in Grothendieck \cite{GRO}, consider the degree zero part of this last localized module. In this section we make the relevant definitions, and compare the different methods.  

\begin{definition}
Let $T$ be a multiplicatively closed subset (MCS) consisting entirely of homogeneous elements of $A$. We'll call this a ``GMCS".  Since $T$ is an MCS we may construct the localization $T^{-1}M$ as usual.  By definition, $T^{-1}M$ is graded by: $(T^{-1}M)_i \doteq  \{ \frac{m}{t} \in T^{-1}M \mid  m \; \mbox{is homogeneous and} \; \deg m - \deg t = i \}$.With this grading, $T^{-1}M$ becomes a graded $T^{-1}A$-module.
In the case where $\p \in Spec(A)$, and $T$ is the set of homogeneous elements of $ A-\p$, we use the notation $M_{[\p]}$ to denote the localization $T^{-1}M$, graded as above.  
\end{definition}

For a GMCS $T$, we'll assume from now on that $1 \in T$ and $0 \notin T$.

The following list of lemmas collect some facts about graded localizations; we leave the proofs to the reader. 
\begin{lemma}  
Let $\p \in Spec(A)$.  The set of homogeneous elements in $A-\p$ is equal to the set of homogeneous elements in $A- \p^*$.  Therefore,  $M_{[\p]} = M_{[\p^*]}$ 
\end{lemma}

\begin{lemma}
\label{lemma graded localize the quotient neq 0}  
Let $\p$ and $\q$ be prime ideals of $A$, with $\q$ graded.  Then, $(A/\q)_{[\p]} \neq 0$ if and only if $\q \subseteq \p^*$. If $\p$ is a minimal prime of $A$, then $(A/\q)_{[\p]} \neq 0$ if and only if $\q = \p$. 
\end{lemma}

\begin{lemma}  \label{lemma props of graded localization}If $M \in \grmod(A)$, and $T$ is a GMCS in $A$, then
\begin{enumerate}
\item[a)] $T^{-1}M \in \grmod(T^{-1}A).$
\item[b)]  If $A$ is a Noetherian ring then $T^{-1}A$ is a Noetherian ring and $T^{-1}M \in \grmod(T^{-1}A)$.
\item[c)]  There is a one-one, inclusion-preserving correspondence between the prime ideals in $A$ that are disjoint from $T$, and the  prime ideals in $T^{-1}A$ given by $\p \mapsto T^{-1}\p$; moreover this correspondence restricts to a one-one correspondence between the graded prime ideals in $A$ disjoint from $T$ and the graded prime ideals in $T^{-1}A$, and further restricts to a one-one correspondence between the ideals (all graded) in $Ass_A(M)$ that are disjoint from $T$, and the ideals (also all graded) in $Ass_{T^{-1}A}T^{-1}M$.
\end{enumerate}
\end{lemma}

\begin{lemma}
\label{lemma graded iso of localizing a shifted graded module}
Let $M \in \grmod(A)$, $T$ a GMCS in $A$, and let $d$ be any integer.  Then there is a graded isomorphism of graded $T^{-1}A$-modules $T^{-1}(M(d)) \cong (T^{-1}M)(d)$.
\end{lemma}

\begin{lemma}  Suppose that $M \in \grmod(A)$.  
\begin{enumerate}
\item[a)]  $\p \in Supp_A(M)$ if and only if $M_{[\p^*]} \neq 0$ if and only if $\p^* \in *V(M)$.  Therefore, $$*V(M) = *Supp_A(M) =
 \{ \q \in *V(A) \mid M_{[\q]  } \neq 0\}.$$
\item[b)]  If $0 \rightarrow M \rightarrow N \rightarrow P \rightarrow 0$ is a short exact sequence in $\grmod{A}$, then $*V(N) = *V(M) \cup *V(P).$

\end{enumerate}
\end{lemma}
\begin{proof}  Since $V(N) = V(M) \cup V(P)$, b) follows.

For $a)$, it's straightforward to see that the ungraded object $M_{\p} \neq 0$ implies that $M_{[\p^*]} \neq 0$.  If $M_{[\p^*]} \neq 0$, and $Ann_A(M)$ is not contained in $\p^*$, then since both are graded ideals, there exists a homogeneous element $r \in Ann_A(M)$ such that $r \notin \p^*$.  But then, $m/t = 0/r = 0$ for every $m \in M$ and homogeneous $t \notin \p^*$.  Finally, suppose that $Ann_A(M) \subseteq \p^*$, yet $M_{\p} = 0$.   If $x_1, \ldots, x_j$ are homogeneous elements of $M$ generating $M$ as an $A$-module, since $x_i/1= 0$ for every $i$, there exist $s_i \notin \p$ such that $s_ix_i=0$ for each $i$.  We may assume that  each $s_i$ is homogeneous, since $x_i$ is.  Since $s_i \notin \p$, $s_i \notin \p^*$, so that $s = s_1s_2\cdots s_j \notin \p^*$ and is homogeneous.  Furthermore, $sm= 0$ for every $m \in M$, so $s \in \p^*$, a contradiction.  
\end{proof}
\begin{lemma}
For $\p$ a graded prime in $A$, $T$ a GMCS, $(T^{-1} \p)_0=T^{-1}\p \cap (T^{-1}A)_0$, and if $\p \cap T = \emptyset$ then $(T^{-1}\p)_0$ is a prime ideal in $(T^{-1}A)_0$.  
\end{lemma}

\begin{definition} 
\cite{GRO} If $\p \in Spec(A)$, then we denote the degree $0$ part of $M_{[\p]}$ by $M_{(\p)}$. 
\end{definition}

If $M$ is an $A$-module,  $M_{(\p)}$ is an $A_{(\p)}$-module.

\begin{example}  If $A$ is a graded ring,  $\p$ is a graded prime ideal in $A$, and $T$ is the GMCS consisting of the homogeneous elements of $A-\p$, then $T^{-1}\p \doteq \p_{[\p]}$ is a *maximal ideal in $T^{-1}A \doteq A_{[\p]}$ and $\p_{(\p)} = (\p_{[\p]})_0$ is a maximal ideal in $A_{(\p)}$.
\end{example}

Now, if $M$ is a graded $A$-module and $\p$ is a graded prime ideal, we know that the standard localization $M_{\p}$ isn't usually graded as we allow inhomogeneous elements of $A$ not in $\p$ to be inverted.  If $\p$ is a minimal prime ideal for $M$, it must be graded, as we have seen, and from ordinary commutative algebra, we know that $M_{\p}$ has finite length as an $A_{\p}$-module.  But we can also consider the graded localization $M_{[\p]}$ and the comparison between length and *length: 

\begin{theorem}
\label{theorem *composition series exists for Mgrp when p is minimal}  Suppose that $A$ is a Noetherian graded ring.
Let $M \in \grmod(A)$, and $\p$ be a prime minimal over the graded ideal $Ann_A(M)$. Then, a *composition series exists for the graded $A_{[\p]}$-module $M_{[\p]}$.  Moreover, $$ *\ell_{A_{[\p]}}(M_{[\p]}) =\ell_{A_{\p}}(M_{\p}).$$
\end{theorem}

\begin{proof}
We will produce a *composition series for $M_{[\p]}$, as an $A_{[\p]}$-module and calculate its length.  

Construct a graded filtration $M^{\bullet}$ as in  Lemma \ref{corollary the filtration graded}, and then localize this filtration using the graded localization.  We now have a filtration of $M_{[\p]}$ by graded $A_{[\p]}$-submodules which looks like $0=(M^{0})_{[\p]} \subseteq (M^1)_{[\p]} \subseteq \cdots \subseteq (M)_{[\p]}$.  By exactness of localization and the condition on successive quotients of $M^{\bullet}$ we have that $(M^{i+1}/M^{i})_{[\p]} \cong (A/\p_{i+1}(-d_{i+1}))_{[\p]}$ is a graded isomorphism of $A_{[\p]}$-modules, for appropriate integers $d_i$, where the graded primes $\p_i$ are  as in \ref{corollary the filtration graded}.

There is a graded isomorphism $((A/\p_i)(-d_i))_{[\p]} \cong (A/\p_i)_{[\p]}(-d_i)$, and 
$$(A/\p_i)_{[\p]}(-d_i) \neq 0 \; \mbox{ if and only if} \;  \p = \p_i$$ (by minimality of $\p$). 

In the case that $\p \neq \p_i$,  $(A/\p_i)_{[\p]} =0$  and we have $(M^{i})_{[\p]} = (M^{i-1})_{[\p]}$. Now throw away all such submodules $(M^i)_{[\p]}$ which are equal to the  submodule $(M^{i-1})_{[\p]}$ to get a reduced filtration $((\overline{M}^j)_{[\p]})$ of $M_{[\p]}$, where for each $j$, $(\overline{M}^{j})_{[\p]} \subset (\overline{M}^{j+1})_{[\p]}$ is a strict inclusion, $(\overline{M}^{s})_{[\p]} = M_{[\p]}$ for some $s$, and the zeroth term of the filtration is zero. The claim is that this reduced filtration forms a *composition series for $M_{[\p]}$ of *length equal to the number of times that $A/\p$, shifted,  appeared as a successive quotient in the original filtration $M^{\bullet}$.

For each $j$ the successive quotient $(\overline{M}^{j+1})_{[\p]} / (\overline{M}^{j})_{[\p]}$ is graded isomorphic  to $(A/\p)_{[\p]}(-d_{j+1})$, as an $A_{[\p]}$-module.  But $(A/\p)_{[\p]}$ is a graded field, since $A_{[\p]}$ has a unique graded prime ideal $\p_{[\p]}$; thus,  $(\overline{M}^{j+1})_{[\p]}/(\overline{M}^{j})_{[\p]} \cong (A/\p)_{[\p]}(-d_{j+1})$ is a *simple $A_{[\p]}$-module for each $j$. 

Going back to the original filtration $M^{\bullet}$ and forgetting the grading everywhere, recall that the number of times that $A/\p$ appears as a successive quotient in any finite filtration of $M$ which has successive quotients isomorphic to $A/\q$ for some prime $\q$, graded or not, is always the same, and is equal to $\ell_{A_{\p}}(M_{\p})$.  

\end{proof}

\subsection{*Local rings}

\begin{definition}  If $A$ is a graded ring, then $A$ is *local if and only if there is one and only one *maximal ideal of $A$.
\end{definition}

Some examples of *local rings are immediate.  For example, a graded field is always *local, with unique *maximal ideal $0$.  This shows that generally, a *maximal ideal of a graded ring $A$ may not be a maximal ideal of $A$.  If $\p$ is a graded prime in $A$, then $A_{[\p]}$ is a *local ring with unique *maximal ideal $\p_{[\p]}$. The end of this section gives a partial characterization of *local rings.

Also, as one might expect, if $A$ is a *local graded ring, with unique *maximal ideal $\mathcal{N}$, then
\begin{itemize}
\item  For every proper ideal  $\mathcal{I}$ (graded or not) of $A$, $\mathcal{I}^* \subseteq \mathcal{N}$.
\item Every homogeneous element of $A-\mathcal{N}$ is invertible:  i.e., for every $x \in A-\mathcal{N}$ with $\deg x = d$, there exists a $y \in A-\mathcal{N}$ of degree $-d$ such that $xy = 1 \in A_0$.
\item $A/\mathcal{N}$ is a graded field; also, for every $y \in \mathcal{N}_j$ and every $x \in \mathcal{A}_{-j}$, $1-xy \in A_0$ is a unit in $A_0$. 
\end{itemize}

\begin{lemma} (Graded Nakayama's lemma) \label{lemma graded nakayama}Suppose $(A,\mathcal{N})$ is a *local ring and $M$ is a finitely generated graded $A$-module with $N$ a graded $A$-submodule of $M$.  If $\q$ is a proper graded ideal in $M$,  then $N+ \q M = M$ implies that $M = N$.
\end{lemma}
\begin{proof} (Slight variation of proof of Nakayama's lemma in \cite{AtMac}.)  We may assume $N= 0$ by passing to $M/N$.  Say $M \neq 0$; choose a homogeneous generating set $x_1, \ldots, x_r$ for $M$ over $A$ with a minimal number $r \geq 1$ of nonzero homogeneous elements.  Suppose that $\q M = M$; then there are homogeneous elements $\alpha_j \in \q \subseteq \mathcal{N}$ such that
$x_r = \alpha_1x_1 + \cdots + \alpha_r x_r$; we must have $\deg \alpha_j + \deg x_j = \deg x_r$ for every $j$ such that $\alpha_j x_j \neq 0$.   By minimality, $\alpha_r x_r\neq 0$ and so  $\deg \alpha_r = 0$.  Using the remarks above, $1-\alpha_r$ is an invertible element of $A_0$.  Thus, we may write $x_r$ as an $A$-linear combination of $x_1, \ldots, x_{r-1}$, contradicting the minimality of $r$.
\end{proof}

\begin{proposition}  \label{proposition maximal equals minimal} If $A$ is *local and Noetherian with unique *maximal ideal $\mathcal{N}$, and $M$ is a nonzero finitely generated graded  $A$-module with $\mathcal{N}$ a minimal prime over $Ann_A(M)$ (equivalently, $*V(M) = \{ \mathcal{N} \}$), then $M$ is a *Artinian $A$-module, and for each $j \in \mathbb{Z}$, $M_j$ is an Artinian $A_0$-module. Furthermore, for each $j \in \mathbb{Z}$,   $\ell_{A_0}M_j \leq *\ell_{A}M.$  If, in addition, there is a homogeneous element of degree 1 (or, equivalently, -1) in $A - \mathcal{N}$, $\ell_{A_0}M_j = *\ell_{A}M$ for every $j$.
\end{proposition}
\begin{proof}  $M$ is *Artinian, since $*V(M) = \{\mathcal{N}\}$.  In fact, in this case, $M$ has a *composition series with the property that each successive quotient is annihilated by $\mathcal{N}$ and is also free of rank one over the graded field $A/\mathcal{N}$.  Taking the degree $j$ part of each module in this *composition series, we get a chain of $A_0$-submodules of $M_j$ and  the dimension of each successive quotient over the field $K \doteq (A/\mathcal{N})_0 \doteq A_0/\mathcal{N}_0$ is either zero or 1.  Thus, since $\mathcal{N}_0$ also annihilates each successive quotient in this ``degree j" filtration, we see that we can make appropriate deletions in the ``degree j" part of the *composition series for $M$ to yield a composition series for $M_j$ over $A_0$ of length less than or equal to $*\ell_A(M)$.

For the last statement, supposing that there is a  homogeneous element of degree 1  in $A-\mathcal{N}$, then there are nonzero elements of every degree in the graded $A$-module $(A/\mathcal{N})(d)$, for every $d \in \mathbb{Z}$; to see this, note that $A/\mathcal{N}$ is a graded field, equal to $K[T, T^{-1}]$, where $T \neq 0$ has least positive degree in $A/\mathcal{N}$, namely degree 1.  So, each successive quotient in the *composition series for $M$ is nonzero in every degree. After taking the ``degree j" part of this *composition series, each quotient must be of rank 1 over $K$.  Thus the equality holds.

\end{proof}

\begin{corollary}  \label{corollary length and *length at minimal prime}Suppose that $A$ is a Noetherian graded ring.
Let $M \in \grmod(A)$, and $\p$ be a prime minimal over $Ann_A(M)$, necessarily graded. Then,  $M_{[\p]}$ is an *Artinian $A_{[\p]}$-module and $M_{(\p)}$ is an Artinian $A_{(\p)}$-module.  Also,
$$\ell_{A_{(\p)}} ( M_{(\p)} )\leq *\ell_{A_{[\p]}}(M_{[\p]}) = \ell_{A_{\p}}(M_{\p}).$$  In addition, if there is a homogeneous element of degree 1 (or -1) in $A-\p$, then 
$$\ell_{A_{(\p)}} ( M_{(\p)} )= *\ell_{A_{[\p]}}(M_{[\p]}) = \ell_{A_{\p}}(M_{\p}).$$
\end{corollary}

In ending this section, we point out that *local graded rings are often graded localizations of positively (or negatively) graded rings at graded prime ideals.  

Suppose that  $(A, \mathcal{N})$ is a *local ring.  Then, there exists a  homogeneous element of strictly positive degree in $A-\mathcal{N}$ if and only if there exists a  homogeneous element of strictly negative degree in $A-\mathcal{N}$:  if $s \in A-\mathcal{N}$ is homogeneous of degree $d >0$ then, since every homogenous element of $A-\mathcal{N}$ is invertible, there exists a $t \in A-\mathcal{N}$ that is homogeneous and $st=1 \in A_0$.  Necessarily, the degree of $t$ is $-d$. Since the argument is reversible, we have the conclusion.

Thus, we have alternatives:
\begin{itemize}
\item  There exist  homogeneous elements of $A-\mathcal{N}$ in at least one strictly positive degree and at least one strictly negative degree.  The analysis of this alternative is given below and we see that $A$ is a graded localization of a positively graded ring at a graded prime ideal.
\item $A$ is a positively or negatively graded ring:  in this case, $A$ is the localization of itself (a positively or negatively graded ring) at the *maximal ideal $\mathcal{N}$ since we've seen that  $\mathcal{N}_0$ is the unique maximal ideal in $A_0$ and $A_e = \mathcal{N}_e$ for all $e \neq 0$.
\item  $A$ has nonzero elements of both positive and negative degree,  $\mathcal{N}_d = A_d$ for all $d \neq 0$ and $\mathcal{N}_0$ is the unique maximal ideal of $A_0$.  In this case, since $\mathcal{N}$ is an ideal, we must have $A_d A_{-d} \subseteq \mathcal{N}_0$, for all $d \neq 0$.    As an example, consider $A = k[s,t]/(st)$ where $k$ is a field (all elements of degree 0), the degree of $s$ is one and the degree of $t$ is -1.  In this case, one might not be able to obtain $A$ as a graded localization of a positively (or negatively) graded object at a graded prime ideal.  But we don't fully analyze this case here. \end{itemize}

Anyway, in the case of the first of the alternatives, let $$S(A) = \oplus_{d \geq 0} A_d$$ be the ``positive part" of $A$, graded with the natural grading; this too is a graded ring, and it is certainly positively graded.  Considering the graded abelian subgroup $S(\mathcal{N}) = \oplus_{d \geq 0} \mathcal{N}_d$ of $S(A)$, we see that it is  a graded prime ideal in $S(A)$.  We claim that $S(A)_{[S(\mathcal{N})]}$ is isomorphic as a graded ring to $A$, with $\mathcal{N}$ corresponding to $S(\mathcal{N})_{[S(\mathcal{N})]}$, under the well-defined injective homomorphism of graded rings defined by  $a/b \mapsto ab^{-1}$, if $a \in S(A)_d$ and $b \in A_e-\mathcal{N}_e$ for $d, e \geq 0$.  To see that the homomorphism is surjective, suppose that $x$ is a homogeneous element of degree $j$ in $A$.  If $j \geq 0$, $x/1 \mapsto x$, and $x/1 \in S(A)_{[S(\mathcal{N})]}$.  If $j<0$, the assumption of the first alternative says that there is a homogeneous element $t \in A-\mathcal{N}$ with $\deg(t) = k >0$.  Then, there is a positive integer $l$ such that $lk+j >0$ so that $t^lx/t^l \in S(A)_{[S(\mathcal{N})]}$ and $t^lx/t^l \mapsto x$.

\section{Krull Dimension in grmod(A)}

\begin{remark} From now on, we assume that $A$ is a Noetherian graded ring, unless explicitly stated otherwise. \end{remark}

\label{Dimension in C(A)}

The height of a prime ideal $\p$ of $A$, graded or not,  is  defined as usual:  $\height({\p})$ is the longest length $n$ (which always exists, using the Noetherian hypothesis) of a chain of primes $\p_0 \subset \cdots \subset \p_n=\p$; thus  we define the graded height $\sheight({\p})$ of a graded prime ideal $\p$ in the ring $A$, as the longest length $m$ (which always exists, using the Noetherian hypothesis) of a chain of graded primes $\p_0 \subset \cdots \subset \p_m=\p$.   For every graded prime $\p$, $\height(\p) \geq \sheight(\p)$.

Forgetting the grading on $A$ and $M$,  one defines the Krull dimension of a graded $A$-module $M$ as usual; here this is denoted by $\dim_A(M)$.  As usual, $\dim(A) \doteq \dim_A(A).$ The \textbf{graded Krull dimension} of a graded $A$-module $M$, $\sdim_A(M)$, is the greatest $D$ such that there exists a strictly increasing chain
$$\mathfrak{p}_0 \subset \ldots \subset \mathfrak{p}_D$$ of graded prime ideals in $A$ such that $Ann_A(M) \subseteq \mathfrak{p}_0$. If no such greatest $D$ exists, $M$ has infinite graded Krull dimension.  For the zero module, we define $\sdim(0) = -\infty.$
By definition, $\sdim(A) \doteq \sdim_A(A).$  For any graded $A$-module $M$,
\begin{itemize}
\item $\sdim_A(M) \leq \dim_A(M)$.
\end{itemize}
Also, since $Ann_A(M) = Ann_A(M(n))$, for every $n \in \mathbb{Z}$,
\begin{itemize}
\item $\dim_A(M) = \dim_A(M(n)), $ for every $n \in \mathbb{Z}$.
\item $\sdim_A(M) = \sdim_A(M(n)),$ for every $n \in \mathbb{Z}$.
\end{itemize}

\begin{example} \label{example graded Krull dimension of Laurent polynomial ring}

Let  $F$ be a graded field of the form $F \cong F_0[t,t^{-1}]$, where $\deg(t)>0$. The only graded prime in $F$ is $0$, so that $\sdim(F)=0$.  On the other hand, $\dim(F)=1.$
\label{laurentex}
\end{example}

More generally, 

\begin{example}If $A$ is has only one graded prime ideal $\mathcal{N}$, then $A$ is a *local, *Artinian ring with $\sdim(A) = 0$, and $A_0$ is an Artinian ring of Krull dimension zero with unique nilpotent maximal ideal $\mathcal{N}_0$.
\end{example}

Most of the proofs for the following lemma may be found in \cite{brhe}.
\begin{lemma} Suppose that $\p \in$ Spec($A$).  ($\p$ may or may not be graded.) We know that $\p$ has finite height; say $\height(\p)=d$.
\label{mainlemma6.2}
\begin{itemize}
\item[i)] If $\q \in Spec(A)$ and $\p^* \subseteq \q \subseteq \p$ then either $\q = \p$ or $\q = \p^*$.
\item[ii)] There exists a chain of primes $\q_0 \subset \cdots \subset \q_d = \p$, such that $\q_0, \cdots , \q_{d-1}$ are all graded.
\item[iii)] If $\p$ is graded then there exists a chain of graded prime ideals such that $$\p_0 \subset \p_1 \subset \cdots \subset \p_d = \p \text{,}$$ so that $\height(\p) = \sheight(\p)$.
\item[iv)] If $\p$ is not graded $(\p^*$ is a proper subset of $\p)$, then $$\height(\p) = \height(\p^*)+1 = \sheight(\p^*)+1 \text{.}$$
\end{itemize}
\end{lemma}
\begin{corollary} \label{corollary difference Krull dim Z graded} If $A$ is a graded Noetherian ring, then $$\sdim(A) \leq \dim(A) \leq \sdim(A) +1;$$therefore if $M \in \grmod(A)$,
$$\sdim_A(M) \leq \dim_A(M) \leq \sdim_A(M) +1.$$
\end{corollary}

\begin{proof}  If $A$ has finite Krull dimension, the first inequality is always true; also, there is a maximal ideal $\m$ of $A$, not necessarily graded, such that $\height(\m) = \dim(A)$.  But then $\dim(A) = \height(\m) \leq \sheight(\m^*) +1 \leq \sdim(A) +1.$  If $A$ does not have finite Krull dimension, then for every positive integer $e$ there is a prime ideal  $\p$ of height  larger than $e$.  But then $\sheight(\p^*)$ is larger than $e-1$, so $\sdim(A) $ is infinite as well.
\end{proof}

\subsubsection{Krull dimension for modules over positively graded rings}
\label{Krull dimension for modules over positively graded rings}

\begin{remark} The results in this section also hold for negatively graded rings, after changing definitions appropriately. \end{remark}

\begin{definition}  If $S$ is a positively graded ring,  $$Proj(S) \doteq \{ \p \in \mbox{Spec}(S) \mid \p \; \mbox{is graded and} \; S_+ \not \subseteq \p \}.$$
  \end{definition}

Note that if $\p \in Proj(S)$, then the set of homogeneous elements of $S-\p$ has at least one nonzero element of strictly positive degree.  We've noted that  $\mathcal{N}$ is a *maximal ideal in $S$ if and only if $\mathcal{N} = \mathcal{N}_0 \oplus S_+$, with $\mathcal{N}_0$ a maximal ideal in $S_0$.  Thus, $Proj(S)$ contains no *maximal ideals.
 
 For positively graded rings, there is no difference between $\sdim$ and $\dim$:
 
\begin{lemma}
Let the ring $S$ be a positively graded ring of finite Krull dimension and $M \in \grmod(S)$.   Then,
\begin{itemize}
\item[i)] $\dim(S) = \sdim(S)$; therefore,
\item[ii)] $\dim_S(M) = \sdim_S(M)$.
\end{itemize}
\label{maintheorem6.2}
\end{lemma}

For graded localizations of positively graded rings, the following is well-known:

\begin{theorem} \label{theorem krull dimension localized Z graded} Suppose that $S$ is a positively graded  ring of finite Krull dimension, and $\p$ is a graded prime ideal of $S$.  Then, if $S_+ \subseteq \p$, $\dim(S_{[\p]}) = \sdim(S_{[\p]})$, and if $S_+ \not \subseteq \p$, $\dim(S_{[\p]}) = \sdim(S_{[\p]}) + 1.$
\end{theorem}

\begin{proof}  Since $S_{[p]}$ is a localization of $S$,  it is Noetherian. Ignoring the grading and recalling the standard order-preserving correspondence between the set of all primes of $S$ disjoint from $T$ and the prime ideals of of $T^{-1}S$, for any MCS or GMCS $T$ in $S$.  So $\infty > \dim(S) \geq \dim(S_{[\p]}).$

We have already seen, then, that $ \sdim(S_{[\p]}) \leq \dim(S_{[\p]}) \leq \sdim(S_{[\p]}) + 1$.
 
 Now let $T$ be the GMCS consiting of all homogeneous elements of $S$ not in $\p$.
 
In the case where $S_+ \subseteq \p$, we must have $\p = (\p \cap S_0) \oplus S_+.$  For any element $t \in T$, this forces $\deg t = 0$.  Thus, $S_{[\p]}$ is a positively graded ring of finite Krull dimension, so $\dim(S_{[\p]}) = \sdim(S_{[\p]}).$

Now,  $S_{[\p]}/\p_{[\p]} = (S/\p)_{[\p]}$ is a graded field, and it does have a positive degree element since $S_+ \not \subseteq \p$:  Choose any homogeneous $t \in S_+$, $t \notin \p$.  Then $t \in T$, and has positive degree, thus $(t+\p)/1$ is a nonzero, positive degree element of $(S/\p)_{[\p]}$.  Forgetting the grading, this domain has dimension 1.  Thus, there must exist a prime $\q$, necessarily ungraded, of $S_{[\p]}$ such that 
$$\p_{[\p]} \subset \q.$$  Therefore 
$$\dim(S_{[\p]}) \geq \height(\p_{[\p]}) +1 = \sheight(\p_{[\p]}) +1= \sdim(S_{[\p]})+1,$$ yielding the conclusion.
\end{proof}

The following  lemma establishes a relationship between primes in the localized ring and primes in the degree $0$ part of the localization, the ideas are implicit in \cite{GRO}. 

\begin{lemma} Suppose that $S$ is a  positively graded ring, and $T$ is any GMCS that contains at least one element of positive degree.
\label{lemma existence of degree 0 localization correspondence}
If $\q$ is a prime ideal in $(T^{-1}S)_0$, then there exists a unique graded prime $\p \in Proj(S)$, disjoint from $T$, such that $\q = (T^{-1}\p)_0$. 
\end{lemma}

\begin{proof}  Uniqueness is left to the reader.  To establish existence, let $\q \in Spec(T^{-1}S)_0$.  Define for $i \geq 0$, 
$$\p_i \doteq \{x \in S_i\; | \; \exists j >0, t \in T_j \text{ s.t. } \frac{x^j}{t^i} \in \q \},$$ so that, since $\q$ is prime,
$$\p_0 =  \{r \in S_0 \; | \;  \frac{r}{1} \in \q \}.$$ Define $\p \doteq \oplus_{i\geq 0} \p_i$, we will show that $\p$ satisfies the required conditions. 

First, each $\p_i$ is an abelian group with respect to $+$.  For if $x,y \in \p_i$, $i  \geq 0$, then there exists a $k_1,k_2 >0$ and $s \in T_{k_1}, t \in T_{k_2}$ such that $\frac{x^{k_1}}{s^i}$ and $\frac{y^{k_2}}{t^i}$ are in $\q$. Then, $(x+y)^{k_1+k_2} = \sum_{\alpha+\beta=k_1+k_2}c_{(\alpha,\beta)}x^\alpha y^\beta$, for the binomial coefficient $c_{(\alpha,\beta)} \in S_0$.  Now, either $\alpha \geq k_1$ or $\beta \geq k_2$.  If $\alpha \geq k_1$, then $\frac{x^\alpha y^\beta}{s^i t^i} = \frac{x^{k_1}}{s^i} \cdot \frac{x^{\alpha-k_1}y^\beta}{t^i}$.  This is a product of an element in $\q$ with an element in $(T^{-1}S)_0$, so it must be in $\q$. A similar computation handles the case that $\beta \geq k_2$.  Therefore, $\frac{(x+y)^{k_1+k_2}}{(st)^i} \in \q$, and so $x + y \in \p_i$.  

To show that  that $\p$ is an ideal in $S$, one needs only to show that $S_i \p_j \subseteq \p_{i+j}$ for every $i,j$.  Suppose $s \in S_i$ and $x \in \p_j$.  There exists $k>0, t \in T_k$ with $\frac{x^k}{t^j} \in \q$.  Then, $\frac{(sx)^k}{t^{j+i}} = \frac{s^k}{t^i} \cdot \frac{x^k}{t^j}$, the product of an element in $(T^{-1}S)_0$ with an element in $\q$, and therefore $sx \in \p_{i+j}$ so that $\p$ is a graded ideal in $S$. 

Furthermore, $\p \cap T = \emptyset$: if not, choose  a $t \in \p_i \cap T$. So, there exists a $k >0$ and an $s \in T_k$ such that $\frac{t^k}{s^i} \in \q$.  However, the product $\frac{s^i}{t^k}\cdot \frac{t^k}{s^i}$ must also be in $\q$, which contradicts that $1 \not \in \q$.  Since $T$ has at least one nonzero element of positive degree, and $\p \cap T = \emptyset,$ $S_+ \not \subseteq \p$.

To verify that $\p$ is prime, suppose that $f \in S_n$, $g\in S_m$, and  $fg \in \p_{n+m}$.  There exists a $k >0$ and $t \in T_k$ such that $\frac{(fg)^k}{t^{m+n}} \in \q$.  Now, $\frac{(fg)^k}{t^{m+n}} = \frac{f^k}{t^n} \cdot \frac{g^k}{t^m} \in \q$, and by primality of $\q$, together with the definition of $\p$,  either $f \in \p_n$ or $g\in \p_m$.   

We have established that $\p \in Proj(S)$, and it only remains to show that $\q = (T^{-1}\p)_0$. Suppose that $\xi \in \q$, so $\xi$ may be written as $\frac{x}{t}$, with $x \in S_i$, $t \in T_i$.   If $i >0$ , then $x^i/t^i = \xi^i \in \q$, so $x \in \p_i$ by definition, and  $ \xi = \frac{x}{t} \in (T^{-1}\p)_0$.  If $i = 0$, then $\frac{t}{1}\xi = \frac{x}{1} \in \q$, so $x \in \p_0$ and $\xi =\frac{x}{t} \in (T^{-1}\p)_0$.

On the other hand, suppose that $\frac{x}{t} \in (T^{-1}\p)_0$, $x \in \p_i$, $t \in T_i$. By definition, there exists a $k >0$, and an $s \in T_k$ such that $\frac{x^k}{s^i} \in \q$.   Then, $\frac{s^i}{t^k}\cdot \frac{x^k}{s^i} \in \q$ since $\frac{s^i}{t^k} \in (T^{-1}S)_0$ and $\frac{x^k}{s^i} \in \q$.  Of course, $\frac{s^i}{t^k}\cdot \frac{x^k}{s^i} = \left(\frac{x}{t}\right)^k$, and by primality of $\q$, $\frac{x}{t} \in \q$. 
\end{proof}

Thus we have

\begin{theorem}  Suppose $S$ is a positively graded ring, and $T$ is a GMCS in $S$ containing at least one element of positive degree. Then, there exists a one-to-one inclusion-preserving correspondence 
$$
\{\p \in \mbox{Proj}(S) \;| \;\p \cap T = \emptyset \} \leftrightarrow
\{\q \in Spec (T^{-1}S)_0 \};
$$
this correspondence takes $\p$ to $(T^{-1}\p)_0.$
\end{theorem}

Using the correspondence of the above theorem, we have, as expected,

\begin{corollary}  Suppose $S$ is a positively graded ing of finite Krull dimension.  Let $\p \in Proj(S)$.  Then $S_{(\p)}$ is a local  Noetherian ring of finite Krull dimension and
$$\dim(S_{(\p)}) = \sheight(\p) = \sheight(\p_{[\p]})= \sdim(S_{[\p]})= \dim(S_{[\p]}) -1.$$
\end{corollary}

\subsubsection{Poincar\'{e} series and dimension for positively graded rings}

The ring $\mathbb{Z}[[t]][t^{-1}]$ is denoted by $\mathbb{Z}((t))$; thus an element of $\mathbb{Z}((t))$ is a formal Laurent series $f(t)$ with integer coeffiicients;  there always exists an $n \in \mathbb{Z}$ with $t^nf(t) \in \mathbb{Z}[[t]].$

In order to define the Poincar\'{e} series for a graded abelian group $M$, we assume, in addition, that 1) each $M_j$ is a finitely generated module over a ordinary commutative Artinian ring $S_0$ and 2)  $M_j = 0$ for $j <<0$.   Whenever we write down a Poincar\'{e} series for a graded abelian group $M$, we will make these assumptions.   

For example, if $S$ is a positively graded Noetherian ring and $M \in \grmod(S)$, then as long as $S_0$ is Artinian, 1) and 2) hold.

\begin{definition} Suppose that $M$ is a graded abelian group and $S_0$ is an Artinian ring satisfying 1) and 2) above.  Then the \textbf{Poincar\'{e} series} of $M$ is the formal Laurent series with integer coefficients
$$P_M(t) = \sum_{i \in \mathbb{Z}} \ell_{S_0}(M_i)t^i,$$
where $\ell_{S_0}(M_i)$ is the length of the finitely generated module $M_i$ over the Artinian ring $S_0$.  
\end{definition}

Sometimes the Poincar\'{e} series is called the Hilbert series, or the Hilbert-Poincar\'{e} series.

\begin{theorem} \label{theorem Hilbert-Serre}(The Hilbert-Serre Theorem) \cite{AtMac} Let $S$ be a positively graded Noetherian ring with $S_0$ Artinian, $M \in \grmod(S)$. Suppose that $S$ is generated as a $S_0$-algebra by elements $x_1, \ldots , x_n$ of positive degrees $d_1, \ldots , d_n$. Then,
$$P_M(t) = \frac{q(t)}{\prod_{i=1}^{n}(1-t^{d_i})},$$
where $q(t) \in \mathbb{Z}[t, t^{-1}]$.

Furthermore, if $M$ has no elements of negative degree,  $q(t) \in \mathbb{Z}[t]$.
\end{theorem}

From now on, we assume that  $S$ is a positively graded Noetherian ring of finite (Krull) dimension, with $S_0$ Artinian.

Some facts to note about Poincar\'{e} series:
\begin{itemize}
\item Let $\hat{S}$ be another positively graded ring. Assume also that  the graded abelian group $M$ is  in   $\grmod(S)$,  $S_0 = \hat{S}_0$ is Artinian and $M$ is also a graded $\hat{S}$-module (but not necessarily finitely generated as such).  Then whether we consider $M$ as an $S$-module or as a $\hat{S}$-module, its Poincar\'{e} series does not change.  For example, let $y_1, \ldots, y_s \in S_+$ be homogeneous.  Define $\hat{S} = S_0\langle y_1, \ldots, y_s \rangle$ to be the graded subring of $S$ generated by $S_0$ and $y_1, \ldots, y_s$.  Now, whether we consider $M$ as an $S$-module, or as an $\hat{S}$-module, its Poincar\'{e} series is the same.
\item  If $M$ has a Poincar\'{e} series with respect to $S$, then so does $M(n)$, for every $n \in \mathbb{Z}$, and
$$P_{M(n)}(t) = t^{-n}P_M(t).$$
\item  If $0 \rightarrow P \rightarrow M \rightarrow N \rightarrow 0$ is a short exact sequence in $\grmod(S)$, then
$$P_M(t) = P_P(t) + P_N(t).$$
\item If $M, N \in \grmod(S)$, then $P_{M \otimes_{S_0}N}(t) = P_M(t)P_N(t),$ if $M \otimes_{S_0} N$ is given the usual grading.
\end{itemize}

We end this section with a brief discussion of the connection of the Poincar\'{e} series with (Krull) dimension.

\begin{definition} Let $M$ be in $\grmod(S)$, $M \neq 0$.
\begin{itemize}
\item If $M \in \grmod(S)$, $d_1(M)$ is the least $j$ such that there exist positive integers $f_1, \ldots, f_j$ with
$$(\prod_{i=1}^j (1-t^{f_i}))P_M(t) \in \mathbb{Z}[t, t^{-1}].$$ By definition, $d_1(M)= 0$ if and only if $P_M(t)$ is in $\mathbb{Z}[t, t^{-1}].$  Note that the Hilbert-Serre theorem shows that $d_1(M) < \infty$; also, $d_1(M)$ is the order of the pole at $t=1$ for $P_M(t).$
\item $s_1(M)$ is the least $s$ such that there exist homogeneous elements $y_1, \ldots, y_s \in S_+$ with $M$ finitely generated over $S_0 \langle y_1, \ldots , y_s \rangle \subseteq S$. By definition, $s_1(M) = 0$ if and only if $M$ is a finitely generated graded $S_0$-module. Note that for a finite set of homogeneous generators for $S_+$, the number of elements in that set is an upper bound for $s_1(M)$.  

\item  $d_1(0)=s_1(0) = -\infty.$
\end{itemize}
\end{definition}
Note that if $n \in \mathbb{Z}$, then $d_1(M(n)) = d_1(M)$, since $P_{M(n)}(t) = t^{-n}P_M(t)$.  Also, $s_1(M(n)) = s_1(M)$ by definition.
The following theorem and proposition could be considered   ``folklore", but the paper of Smoke cited is, as far as we know, the first appearance of these statements in the literature.
\begin{theorem} \textbf{Smoke's Dimension Theorem} (Theorem 5.5 of \cite{sm})\\ \label{theorem Smoke dimension}  Suppose that $S$ is a positively graded ring of finite Krull dimension, with $S_0$ Artinian.
Let $M \in \grmod(S)$.  If $d_1(M), s_1(M)$ are defined as above, we have $$d_1(M) = s_1(M) = \sdim_S(M)  < \infty \text{ .} $$

Under the hypotheses of the theorem, we've already seen that $\sdim_S(M) = \dim_S(M)$, so all of these numbers equal $\dim_S(M)$ as well.

\end{theorem}
\section{Graded ideals of definition and graded systems of parameters}
Returning to the more general case a graded ring $A$, not necessarily positively graded, we define analogously to Serre, a graded ideal of definition and a graded system of parameters.  
\begin{definition} Let $A$ be a graded ring and $M \in \grmod(A)$. 
A proper, graded ideal $\I$ of $A$ such that $*\ell_A(M/\I M) <\infty$ is called a graded ideal of definition for $M$ (a GIOD for $M$). 
\end{definition}
(This is a little different from Serre's definition \cite{se} of an ideal of definition in the ungraded case.)  Lemmas 2.4 and 3.9 say that $\I$ is a graded ideal of definition for $M$ if and only if all graded primes containing $\I + Ann_A(M)$ are  *maximal.

\begin{definition}  Let $A$ be a Noetherian ring of finite Krull dimension, and also assume that $A$ is either a positively graded ring or a *local  ring with unique *maximal ideal $\mathcal{N}$. Define $\m$ to be the graded ideal $A_+$ in the first case, and the ideal $\mathcal{N}$ in the second.
Suppose $M \neq 0$ is in $\grmod(A)$.  A sequence $y_1, \ldots , y_D$ of homogeneous elements of $\m$, such that 
\begin{itemize}
\item  the graded $A$-module $M/(y_1, \ldots, y_D)M$ has finite *length over $A$ and
\item $D = \sdim_A(M)$
\end{itemize}
is called a graded system of parameters (GSOP) for the $A$-module $M$.
\end{definition}

Note that by definition, a GSOP (or a GIOD) for $M$ is also a GSOP (resp. GIOD) for $M(n)$, for every $n \in \mathbb{Z}$ (and vice versa).

In the positively graded case, an alternative characterization  of some GIODs (and thus some GSOPs) is given by:

\begin{lemma}  \label{lemma finite generation gsop} Suppose that $S$ is a positively graded Noetherian ring of finite Krull dimension, with $S_0$ Artinian,  and $y_1, \ldots, y_u$ are homogeneous elements of $S_+$. Let $M \in \grmod(S)$.  Then,  $M/(y_1, \ldots, y_u)M$ has finite *length over $S$ if and only if $M$ is a finitely generated  $S_0\langle y_1, \ldots, y_u \rangle$-module.
\end{lemma}

\begin{proof} Recall that $S_0\langle y_1, \ldots, y_u \rangle$ is the subring of $S$ generated by $S_0$ and $y_1, \ldots, y_u$.   Let $t_0 \in \mathbb{Z}$ be chosen such that $M_j = 0$ for $j < t_0$.  Suppose that $X \doteq M/(y_1, \ldots, y_u)M$ has finite *length over $S$.  We've seen that there exists an integer $t_0 \leq  t_1$ such that $X_j = 0$ if  $j > t_1$.  Using Lemma \ref{remark *Samuel equals Samuel},  $M_j$ is finitely generated over $S_0$,  so for every $j$ such that $t_0 \leq j \leq t_1$ we may choose  a finite set $E_j$ of generators, possibly empty,  for $M_j$ over $S_0$.  Then, we prove that $M$ is generated by the finite set $E \doteq \cup_{j=t_0}^{t_1} E_j$ over $S_0\langle y_1, \ldots, y_u \rangle$; to do this we show, using induction on $\deg(z)$, that a homogeneous element $z$ of $M$ is in the submodule of $M$ generated by $E$ over $S_0\langle y_1, \ldots, y_u \rangle$. To start the induction, note that if  $\deg(z) \leq t_1$, the claim  is certainly true.  Let $s > t_1$ and suppose that the inductive hypothesis holds for every homogeneous $w$ of degree strictly less than $s$.  Let $z$ be a homogeneous element of $M$ of degree $s$.  Since $s > t_1, (M/(y_1, \ldots, y_u)M))s = 0$, so $s \in (y_1, \ldots, y_u)M.$  Write $z = \sum_{\alpha = 1}^{u} y_{\alpha} m_{\alpha}$.  Since $\deg(y_{\alpha}) + \deg(m_{\alpha}) = s$ for every $\alpha$ such that $y_{\alpha}m_{\alpha} \neq 0$, and $\deg(y_{\alpha}) >0$ for every such $\alpha$, we must have $\deg(m_{\alpha})<s$ for every  $\alpha$ with $y_{\alpha}m_{\alpha} \neq 0$.  Thus by induction, $m_{\alpha}$ is a linear combination of elements of $E$, with coefficients in $S_0\langle y_1, \ldots, y_u \rangle$.  Clearly, then, so is $z$.  Note that this part of the proof never used that $S_0$ is Artinian.

Conversely, suppose $M$ is generated by a finite set $E$ of nonzero homogeneous elements as a  graded $S_0\langle y_1, \ldots, y_u \rangle$-module.  Set $\I \doteq (y_1, \ldots, y_u)$.  Let $t = \max \{ \deg(e) \mid e \in E \} $.  Then, for $j > t$, $(M/\I M)_j = 0$:  If $x \in M_j$, $j >t$, write $x = \sum_{e \in E} f_e e$, where $f_e \in S_0\langle y_1, \ldots, y_u \rangle$ is homogeneous.  If $\deg(f_e) \neq 0$, and $f_e e \neq 0$, then $f_e e \in \I M$.  Therefore, $x$ is equivalent to $
\sum_{f_e \neq 0, \deg(f_e) = 0} f_e e $ mod $\I M$.  However, for every summand  in this last sum,  we must have $\deg(e) = \deg(x) >t$ if $f_e e \neq 0$, a contradiction.  Thus, $x$ is equivalent to $0$ mod $\I M$.  Lemma  \ref{lemma *length of components} tells us that since $S_0$ is Artinian,  $*\ell_S(M/\I M) < \infty$.
\end{proof}

The following proposition is another part of the ``folklore" knowledge, but the citation is the first that we know of in the literature.  

\begin{proposition} (Theorem 6.2 of \cite{sm})  \label{proposition algebraic independence} Suppose that $S$ is a positively graded Noetherian ring of finite Krull dimension, with $S_0$ Artinian, and $M \neq 0$ is in $\grmod(S)$. Let  $D(M)\doteq d_1(M) = s_1(M) = \sdim_S(M) = \dim_S(M)$, so  Theorem \ref{theorem Smoke dimension} and Lemma \ref{lemma finite generation gsop} tell us that a GSOP exists for $M$.  Moreover, $D(M)$ is the length of any GSOP and if $y_1, \ldots ,y_{D(M)} \in S_+$ is a GSOP for $M$,  $y_1, \ldots, y_{D(M)}$ are algebraically independent over $S_0$.
\end{proposition}

\section{Multiplicities for graded modules}
In this section,  we define the  *Samuel multiplicity and *Koszul multiplicity for modules in  $\grmod(A)$. All of this work is done analogously to Serre \cite{se}, and since we only give brief discussions/proofs here, if the reader does not have in mind the development of multiplicities in \cite{se}, it's advised to have a copy of \cite{se} at hand.  Another treatment of multiplicity in the graded case is given in \cite{pr}.

The *Samuel multiplicity is explored using the tools of the graded category which we have developed thus far: *length, *dimension, graded localization, etc. The *Koszul multiplicity is defined using tools from homological algebra.  In each case, to adapt the theory from the ungraded case, we have the added complication of our objects being bi-graded - the internal grading that the module inherits from $\grmod(A)$, and an external grading coming from either the associated graded module in the case of Samuel multiplicities, or the complex grading for Koszul multiplicities. Keeping track of the bi-grading,  all morphisms respect both gradings, and as one might expect, the bi-grading does not cause any problems.  We show, as in the ungraded case, the two multiplicities (*Koszul and *Samuel) agree. 

Finally, we show that  the graded multiplicity theory agrees with the ungraded theory by simply forgetting the grading, when we work over positively graded rings. This is to be expected, for we have shown that *length and length agree in the positively graded case.  

In the following, we will consider filtrations of $A$-modules; as previously, we will use upper indices for filtrations, whether working in a graded or an ungraded category.   We'll use notations like $M^{\bullet}$ or often $\mathcal{F}(M)$ for filtrations of $M$ by $A$-modules.  Filtrations will be indexed in different ways, according to convention.
\begin{definition} Suppose that $\I$ is an ideal in $A$.  A filtration $\mathcal{F}(M)$ with $\mathcal{F}^{i+1}(M) \subseteq \mathcal{F}^i(M)$ for every $i \geq 0$, is called $\I$-bonne  if $\I\mathcal{F}^n(M) \subseteq \mathcal{F}^{n+1}(M)$, for every $n\geq 0 $, and with equality for $n>>0$.  
\end{definition}
\begin{example}  If $\I$ is an ideal in $A$, the $\I$-adic filtration $\cdots \subseteq \I^{j+1}M \subseteq \I^jM \subseteq \cdots \subseteq \I M \subseteq M$ is $\I$-bonne.\end{example}

If $A$ is a graded ring, and $M$ a graded $A$-module, a filtration $\mathcal{F}(M)$ is graded if and only if all the submodules $\mathcal{F}^j(M)$ are graded submodules; if $\I$ is a graded ideal, the definition of an $\I$-bonne graded filtration remains the same as in the ungraded case.

\subsection{The Ungraded Case}

We begin by outlining the procedure for defining the Hilbert and Samuel polynomials in the ungraded case (see \cite{se} for full discussion/proofs).  
 
Suppose that $H$ is a positively graded ring with $H_0$ Artinian, and that $H$ is generated as an $H_0$-algebra by a finite number of homogeneous elements $x_1, \ldots, x_u$ in $H_1$. Such a ring $H$ is then called a ``standard" graded ring.   For any finitely generated, positively graded $H$-module $M$,  $M_n$ is a finitely generated $H_0$-module for every $n$.  Since $H_0$ is Artinian, the Hilbert function, $n \mapsto \ell_{H_0}(M_n)$, is defined for all integers $n \geq 0$.   Using induction on the number of generators for $H$ as an $H_0$-algebra, and the additivity of length over exact sequences, one may prove that the Hilbert function is polynomial-like; in other words there is a unique polynomial $f$ with rational coefficients such that $f(n) = \ell_{H_0}(M_n)$ for all $n$ sufficiently large.  The polynomial describing the function $n \mapsto \ell_{H_0}(M_n)$ is called the Hilbert polynomial of $M$ (over $H$).

Recall the delta notation from the theory of polynomial-like functions:  if $f$ is a function with an integer domain, then $\Delta f$ is the function defined by $\Delta f(n) \doteq f(n+1)-f(n)$.  Then, we know that $f$ is polynomial-like if and only if $\Delta f$ is polynomial-like.  We may iterate  the operator ``$\Delta$" on integer domain functions, obtaining operators $\Delta^r$, for $r \geq 0$.

We review the definition of a Samuel polynomial and Samuel multiplicity in the ungraded case. So, for this and the next two paragraphs, suppose that $A$ is an ungraded Noetherian ring, $M$ an ungraded finitely generated $A$-module, and $\I$ is an ideal of $A$ such that $M/\I M$ has finite length over $A$; this last is true if and only if  $V(\I + Ann_A(M))$ consists of a finite number of maximal ideals in $A$.

Summarizing the discussion in \cite{se}, given  an ideal $\I$  with $\ell_A(M/\I M) < \infty$ and an $\I$-bonne filtration $\mathcal{F}(M)$, $\ell_A(M/\mathcal{F}^n(M))$, is well-defined.  Now, $V(M/\I M) = V(Ann_A(M) + \I)$  consists of a finite number of maximal ideals; without loss of generality we may assume that $Ann_A(M) = 0$ and $V(M/\I M) = V(\I)$ consists of a finite number of maximal ideals, so that $A/\I$ is an Artinian ring.  The positively graded associated graded module $gr(M) = \oplus_{n \geq 0} \mathcal{F}^n(M)/\mathcal{F}^{n+1}(M)$ is finitely generated over the positively graded associated graded ring $gr(A) = \oplus_{n \geq 0} \I^n/\I^{n+1}$. Furthermore $gr(A)$ is generated over $gr(A)_0 = A/\I$, an Artinian ring, by elements of degree one, and the Hilbert  polynomial for $gr(M)$ as a $gr(A)$-module exists.  

Then, $n \mapsto \ell_A(M/\mathcal{F}^{n+1}(M)) -  \ell_A(M/ \mathcal{F}^n(M)) = \ell_A(\mathcal{F}^n(M)/\mathcal{F}^{n+1}(M))$ is polynomial-like, and the general theory of polynomial-like functions tells us that the Samuel function $n \mapsto \ell_A(M/\mathcal{F}^n(M))$ is also polynomial-like.  The polynomial describing this function is called the Samuel polynomial $p(M, \mathcal{F}, n)$ of the $A$-module $M$ with respect to the filtration $\mathcal{F}$ and the ideal $\mathcal{I}$.
 
\subsection{The Graded Case}

We make  new, similar definitions in the graded category, now assuming $A$ is a graded Noetherian ring and $M \in \grmod(A)$.  We do not assume that $A$ is positively graded, nor that it is generated by elements of degree 1.  

 To define the *Hilbert polynomial, start  with certain  bigraded objects:  Suppose that $H$ is a bigraded ring such that $H_{i,j} = 0$ for $i <0$, $H_{0,*} \doteq \oplus_{j \in \mathbb{Z}} H_{0,j}$ is a graded ring that is *Artinian and $H$ is generated as an bigraded algebra  over the graded ring $H_{0,*}$ by a finite number of elements in $H_{1,*} \doteq \oplus_{j \in \mathbb{Z}} H_{1,j}$.  $M$  is taken to be a bigraded $H$-module such that $M_{i,j} = 0$ for $i<0$ and $M$ is generated as an $H$-module by a finite number of bi-homogeneous elements.  Then, for each $k \geq 0$, $M_{k,*} \doteq \oplus_{j  \in \mathbb{Z}} M_{k,j}$ is a finitely generated graded $H_{0,*}$-module, so $*\ell_{H_{0,*}}(M_{k,*})$ is well-defined for every $k \geq 0$. Furthermore, the function $k \mapsto *\ell_{H_{0,*}}(M_{k,*})$ is polynomial like.  To see this,  following the argument in \cite{se} for the ungraded case,   use induction on the number of bihomogeneous generators (taken from $H_{1,*}$) for $H$ as an $H_{0,*}$-algebra, and additivity of $*\ell$ over exact sequences of graded modules.  The exact sequence used in Theorem II.B.3.2 of \cite{se} becomes an exact sequence of graded modules, with middle map multiplication by a generator of bidegree $(1,d)$:  $$0 \rightarrow N_{n,*} \rightarrow M_{n,*}(-d) \rightarrow M_{n+1,*} \rightarrow R_{n+1,*} \rightarrow 0;$$ there is a shift for the second graded degree in the second term, and the rest of proof is the same otherwise with length replaced by *length.  Furthermore, using the argument of Theorem II.B.3.2 of \cite{se} for the ungraded case, we see that if $H$ is generated as a bigraded algebra over $H_{0,*}$ by $r$ elements of bidegree $(1,-)$, then the *Hilbert polynomial has degree less than or equal to $r-1$.

In the same spirit, we define a *Samuel function by making appropriate changes to consider the grading, as follows.  

Suppose $A$ is a graded ring, $\I$ is a graded ideal in $A$ and $\mathcal{F}(M)$ is a graded $\I$-bonne filtration of $M$.

Note that if $\I$ is a graded ideal in $A$,  $\mathcal{F}(M)$ is a graded $\I$-bonne filtration of $M$, and $d \in \mathbb{Z}$ is a fixed integer, we may shift degrees by $d$ throughout the filtration yielding an $\I$-bonne filtration $\mathcal{F}(d)$ of $M(d)$:  $\mathcal{F}(d)^n(M(d)) \doteq (\mathcal{F}^n(M))(d)$.  To see that this filtration is also $\I$-bonne,  just compute that $\I (\mathcal{F}(d)^n(M(d))) = (\I \mathcal{F}^n( M))(d)$ as follows.   Suppose that $x \in (\I \mathcal{F}^n(M))(d)_j =( \I \mathcal{F}^n(M))_{d+j}$, so that $x = \sum_t \alpha_t m_t$, where $\alpha_t \in \I, m_t \in \mathcal{F}^n(M)$ are all homogeneous and $\deg(\alpha_t) + \deg(m_t) = d+j$ whenever $\alpha_tm_t \neq 0$.  Thus, $\deg(m_t) = d + (j-\deg(\alpha_t))$ for all such $t$, so that $m_t \in (\mathcal{F}(d)^n)(M(d))_{j-\deg(\alpha_t)}$, $\alpha_t m_t \in \I(\mathcal{F}(d)^n(M(d)))_j$ for every $t$ and $x \in \I(\mathcal{F}(d)^n(M(d)))_j.$   The converse is similarly proved.  In particular, the $d$-suspension of the $\I$-adic filtration on $M$ is the $\I$-adic filtration on $M(d)$.

 Given a GIOD $\I$ for $M$, and a graded $\I$-bonne filtration $\mathcal{F}(M)$, $*\ell_A(M/\mathcal{F}^n(M)) < \infty$. Passing without loss of generality to the case $Ann_A(M) = 0$ as in the ungraded case, we see that $A/\I$ is a *Artinian ring and that the associated bigraded  module $gr(M) = \oplus_{n \geq 0} \mathcal{F}^n(M)/\mathcal{F}^{n+1}(M)$, where $gr(M)_{n,j} \doteq \mathcal{F}^n(M)_j/\mathcal{F}^{n+1}(M)_j$, is finitely generated over the associated bigraded  ring $gr(A) = \oplus_{n \geq 0} \I^n/\I^{n+1}$ (where $gr(A)_{n,j} \doteq ( \I^n)_j/(\I^{n+1})_j$).  Note that $gr(A)$ is generated  by elements of bidegree $(1,-)$, as an algebra over the *Artinian graded ring $A/\I $ and thus the *Hilbert  polynomial for $gr(M)$ as a $gr(A)$-module exists.

\begin{definition}Suppose that $\I$ is a GIOD for $M \in \grmod{A}$ and $\mathcal{F}$ is an $\I$-bonne filtration of $M$.  The *Samuel function with respect to $\mathcal{F}$ and $\I$ is defined on the nonnegative integers  by $n \mapsto *\ell_A(M/\mathcal{F}^{n}(M))$. 

\end{definition}

Since $*\ell_A(M/\mathcal{F}^{n+1}(M)) - *\ell_A(M/\mathcal{F}^n(M)) = *\ell_A(\mathcal{F}^n(M)/\mathcal{F}^{n+1}(M))$, the $\Delta$ operator applied to the *Samuel function is polynomial-like, so

\begin{lemma}  If $M \in \grmod(A)$ and $\I$ is a GIOD for $M$,
the *Samuel function for the graded $\I$-bonne filtration $\mathcal{F}(M)$  is polynomial-like.

\end{lemma}

To set notation, the polynomial that calculates $*\ell_A(M/\mathcal{F}^n(M))$ for $n >>0$ will be called $*p(M,\mathcal{F},n)$, and if $\mathcal{F}$ is the $\I$-adic filtration on $M$, we will instead write $*p(M,\I,n).$ 

The following lemma incorporates graded versions of results in II.B.4 of \cite{se}.

\begin{lemma} \label{lemma Samuel props} Suppose that $M \in \grmod(A)$  and $\mathcal{F}(M)$ is a graded $\I$-bonne filtration of $M$ for some GIOD $\I$ for $M$.  Then
\begin{enumerate}
\item[a)] For every $d \in \mathbb{Z}$, $\I$ is a GIOD for $M(d)$, $\mathcal{F}(d)(M(d))$  is an $\I$-bonne filtration of the graded $A$-module $M(d)$ and 
$*p(M(d), \mathcal{F}(d), n) = *p(M,\mathcal{F},n)$.
\item[b)]  $*p(M, \I,n) = *p(M, \mathcal{F},n) + R(n)$, where $R$ is a polynomial with nonnegative leading coefficient and degree strictly less than that of the degree of $*p(M, \I,n)$.
\item[c)]  If $(Ann_A(M) + \I)/Ann_A(M)$ is generated by $r$ homogeneous elements, then the degree of $*p(M, \I,n)$ is less than or equal to $r$, and $\Delta^r(*p)$ is a constant less than or equal to $*\ell_A(M/\I M).$
\item[d)]  If $0 \rightarrow N \rightarrow M \rightarrow P \rightarrow 0$ is a short exact sequence in $\grmod(A)$, and $\I$ is a GIOD for $M$, then $\I$ is a GIOD for both $N$ and $P$ and 
$$*p(M, \I, n) + R(n) = *p(N, \I, n) + *p(P, \I, n),$$ where $R$ is a polynomial with nonnegative leading coefficient and degree strictly less than that of $*p(N, \I,n)$.
\item[e)]  If $\I$ and $\hat{\I}$ are two GIODs for $M$ such that $*V(\I + Ann_A(M)) = *V(\hat{\I} + Ann_A(M))$, then the degree of $*p(M, \I,n)$ equals the degree of $*p(M, \hat{\I}, n).$
\end{enumerate}
\end{lemma}
\begin{proof}  We've already noted that $\I (\mathcal{F}(d)^n(M(d))) = (\I \mathcal{F}^n( M))(d)$; so that $\mathcal{F}(d)(M(d))$ is an $\I$-bonne filtration of $M(d)$.  The *Samuel polynomials are identical since $M(d)/\mathcal{F}(d)^n(M(d)) = M(d)/ (\mathcal{F}^n(M)(d))=( M/(\mathcal{F}^n(M))(d),$ for every $n$.  The proofs of b)-e) follow exactly the proofs in Section II.B.4 of Lemma 3 and Propositions 10 and 11 of \cite{se}, adapted with clear notational changes to the graded case, and  are not given here.
\end{proof}

Since we will be interested in the leading coefficient of *Samuel polynomials,  b) above tells us that we may as well just consider $\I$-adic filtrations and suppress all talk about $\I$-bonne filtrations; the need to consider general $\I$-bonne filtrations $\mathcal{F}$ is indicated in the proof of d), even though we haven't given it, since the proof of d) uses the Artin-Rees lemma, which also holds in the graded context.

\begin{definition} Suppose that  $M \in \grmod(A)$, $\I$ is a GIOD for $M$ 
and $d \in \mathbb{Z}$, $d \geq \deg(*p(M,\I,n))$. The \textbf{*Samuel multiplicity} of $M$ with respect to $\I$ is defined as $$*e(M,\I,d) \doteq\Delta^d (*p(M,\I,n)).$$  \end{definition}

By properties of the finite difference operator $\Delta$, we see that $*e(M,\I,d) = 0$ whenever $d > \deg(*p(M,\I,n))$.  When $d= \deg(*p(M,\I,n))$, $*e(M,\I,d)$ is a positive integer, and one may compute that $$*p(M,\I,n) = \frac{*e(M,\I,d)}{d!} n^{d} + \text{lower order terms.}$$

Using Lemma \ref{lemma Samuel props}d), we see that if $$0 \rightarrow N \rightarrow M \rightarrow P \rightarrow 0$$ is a short exact sequence in $\grmod(A)$, $\I$ is a GIOD for $M$ and $d \geq \deg(*p(M, \I, n))$, then both $*e(N,\I, d)$ and $*e(P,\I,d)$ exist and $$*e(M, \I, d) = *e(N, \I, d) + *e(P, \I,d).$$

Therefore, using Lemma \ref{lemma Samuel props}a) as well,  we have

\begin{corollary}  \label{corollary sum decomp for multiplicity}Suppose that $M \in \grmod(A)$, $\I$ is a GIOD for $M$ and $M^{\bullet}$ is a graded filtration of $M$ such that $0 = M^0 \subset M^1 \subset \cdots M^{N-1} \subset M^N = M,$ and, for each $N \geq i \geq 1$, there are graded prime ideals $\p_i$ in $A$, integers $d_i$ and graded isomorphisms of $A$-modules $(A/\p_i )(d_i) \cong M^{i}/M^{i-1}$.   Then,

\begin{itemize}
\item[i)]  $\I$ is a GIOD for $A/\p_i$  and $*p(A/\p_i, \I,n)$ exists, for $1 \leq i \leq N$.
\item[ii)]  If $D \doteq \max\{ \deg(*p(A/\p_i, \I, n)) \doteq d_i \mid 1 \leq i \leq N\}$ and $\mathcal{D}(M^{\bullet}) \doteq \{\p_j \mid d_j = D \}$,
$$*e(M, \I, D) = \sum_{\p \in \mathcal{D}(M^{\bullet})} n_{\p}(M^{\bullet}) (*e(A/\p, \I, D)),$$ where $n_{\p}(M^{\bullet})$ is equal to the number of times $A/\p$, possibly suspended, occurs as an $A$-module isomorphic to a subquotient of the filtration $M^{\bullet}$.  Furthermore,  all of the integers on both sides of the equation are strictly positive.
\end{itemize}
\end{corollary}

Finally, we point out some scenarios in which *Samuel multiplicities equal those computed in the ungraded category.

\begin{theorem} \label{theorem star e equals e, pos graded} \textbf{The positively graded case.}  Suppose that $S$ is a positively graded Noetherian ring with $S_0$ Artinian, $M \in \grmod(S)$ and  $\I$  a GIOD for $M$.  Then, the ``ungraded" Samuel polynomial $p(M, \I, n)$ exists,  $ *p(M, \I, n) = p(M, \I, n)$  and, for every $d$, $*e(M, \I, d) = e(M, \I, d).$  
\end{theorem}
\begin{proof}   Lemma \ref{lemma pos graded length equals *length} tells us that, when we forget the grading, $\I$ has the property that $\ell_S(M/\mathcal{I}^nM) = *\ell_S(M/\mathcal{I}^nM) < \infty$. Therefore,  the ``ungraded" Samuel polynomial $p(M, \I, n)$ exists ($p(M, \I, n)$ is computed after forgetting the grading) and  $ *p(M, \I, n) = p(M, \I, n)$ .  So, if $d \geq \deg(*p(M, \I,n)) = \deg(p(M, \I, n))$, $*e(M, \I, d)$ is the exact same multiplicity $e(M, \I, d)$ defined in \cite{se}, after forgetting the grading.
\end{proof}

\begin{theorem}  \label{theorem star e equals e, element of degree 1} \textbf{The *local case in which $A-\mathcal{N}$ has a homogeneous element of degree 1.} Suppose that $(A, \mathcal{N})$ is a *local Noetherian ring,  $M \in \grmod(A)$ and $A-\mathcal{N}$ has a homogeneous element of degree 1.  Then, $\I_0$ is such that $\ell_{A_0}(M_0/(\I_0)^n M_0) < \infty$, $*p(M, \I,n) = p(M_0, \I_0 ,n)$ and for every $d$, $e(M_0, \I_0,d) = *e(M, \I, d).$
\end{theorem}
\begin{proof} First, note that for any graded ideal  $\mathcal{J}$ in $A$,  and every $X \in \grmod(A)$, it turns out in this case that $(\mathcal{J}X)_0 = \mathcal{J}_0 X_0$:  the containment ``$\supseteq$" is clear.  For the remaining containment, let $T \in A-\mathcal{N}$ be any element of  degree one.  Now, every element of $(\mathcal{J}X)_0$ has the form $\sum_j a_jx_j$ where $a_j \in \mathcal{J}$ and $x_j \in X$ and $\deg(a_j) + \deg(x_j) = 0$ whenever $a_jx_j \neq 0$. However, $T$ is invertible in $A$, and $\sum_j a_jx_j = \sum_j (a_j T^{-\deg(a_j)})(T^{\deg(a_j)}x_j) \in \mathcal{J}_0X_0$.  Using this result for powers of $\mathcal{J}$, and induction, we see that 
$$(\mathcal{J}^n)_0 = (\mathcal{J}_0)^n,$$ for every $n \geq 1$.

Recall that if $X \in \grmod(A)$ is such that $*V(X) = \{\mathcal{N} \}$ (or equivalently, $*\ell_A(X) < \infty$), since  $A-\mathcal{N}$ has a homogeneous element of degree 1,  Proposition \ref{proposition maximal equals minimal} tells us that for every $j$, $\ell_{A_0}(X_j) = *\ell_{A}(X)$ for every $j$.  

Putting all this together, if $\I$ is a GIOD for $M$,  and $X = M/\I^nM$, then we have $X_0 = (M/\I^nM)_0 = M_0/(\I^n)_0 M_0 = M_0/(\I_0)^nM_0$ and thus $\ell_{A_0}(M_0/\I_0^nM_0) = *\ell_A(M/\I^nM) < \infty$ for every $n$.  Therefore, $\I_0$ is an ideal such that  the ordinary Samuel polynomial $p(M_0, \I_0, n)$, constructed in the ungraded case for the $A_0$-module $M_0$,  is defined and   $*p(M, \I,n) = p(M_0, \I_0, n)$.  Therefore, in this case, for $d \geq \deg(*p(M, \I,n)) = \deg(p(M, \I, n))$, $*e(M, \I, d)$ is equal to the  multiplicity $e(M_0, \I_0, d)$ defined  in the ungraded case.

\end{proof}

We do not make a comparison if $(A, \mathcal{N})$ is a *local Noetherian ring with no homogeneous elements of degree 1 in $A-\mathcal{N}$.
\subsection{*Dimension, *Samuel polynomials and GSOPs for *local rings}
In this section, $A$ is a *local Noetherian graded ring with unique *maximal graded ideal $\mathcal{N}$. Here we present an analogue in the graded category to the fundamental theorem of dimension theory for local rings.  This theorem shows the relationship between *Krull dimension, graded systems of parameters, and the degree of the *Samuel polynomial.  Applying the results to the category $\grmod(R)$, for $R$ positively graded and $R_0$ a field, we combine the fundamental dimension theorem for *local rings to Smoke's dimension theorem (\ref{theorem Smoke dimension}). In this case, the order of the pole of the Poincare series at $t=1$, equals the measures from the fundamental *local dimension theorem, which in turn equal the ungraded Krull dimension. This is summarized in corollary \ref{corollary fun thm dimension theory pos graded}.  Returning to the theory of multiplicities, we conclude the section with a sum decomposition of the *Samuel multiplicity by minimal primes (corollary \ref{corollary *Samuel sum decomposition}). 

We start by supposing that  $\mathcal{I}$ is a GIOD for $M$; since $A$ is *local, we've seen that this is true if and only if $*V(M/\I M) = \{ \mathcal{N} \}.$   The previous section shows that the degree of the *Samuel polynomial of $M$ with respect to the $\I$-adic filtration does not depend on the choice of $\I$.  We call this degree $*d(M)$. Of course, $\mathcal{N}$ is always a GIOD for $M$.

If $M \in \grmod(A)$, $M \neq 0$, $*s(M)$ is defined to be the least $s$ such that there exist homogeneous elements $w_1, \ldots , w_s \in \mathcal{N}$ such that  the graded $A$-module $M/(w_1, \ldots , w_s)M$ has finite *length over $A$. Note that $*s(M) = 0$ if and only if  $*\ell_A(M) < \infty$. 

The fundamental theorem for *local dimension theory is:  

\begin{theorem} \label{theorem fun thm graded dimension theory} If $(A,\mathcal{N})$ is a *local Noetherian ring and $M \in \grmod(A)$, then
$$*\dim_A(M) =  *d(M) = *s(M).$$
\end{theorem}

\begin{proof}  The proof of this mimics the proof of the analogous theorem in the ungraded, local case given in \cite{se} in Section III.B.2, Theorem 1, but we give a sketch anyway.  First, if $x$ is a homogeneous element of $\mathcal{N}$, let $_xM $ be the graded $A$-module consisting of all elements $m$ of $M$ such that $xm=0$.  If $\deg(x) = d$, then there are short exact sequences in $\grmod(A)$
$$0 \rightarrow _xM(-d) \rightarrow M(-d) \stackrel{\cdot x}{\rightarrow} xM \rightarrow 0,$$ 
$$0 \rightarrow xM \rightarrow M \rightarrow M/xM \rightarrow 0.$$  If $\I$ is a GIOD for $M$, it is also a GIOD for every module in the exact sequences above.  Furthermore, the short exact sequences and Lemma \ref{lemma Samuel props} say that  $*p(_xM, \I, n) - *p(M/xM, \I, n)$ is a polynomial of degree strictly less than $*d(M)$.  It's straightforward to see that $*s(M) \leq *s(M/xM)+1$.

We may as well assume that the GIOD we are using to calculate $*d(M)$ is $\mathcal{N}$.

Next, set $\mathcal{D}(M)$ to be the (finite) set of all  $\p$ in $*V(M)$ with the property that $\sdim_A(M) = \sdim_A(A/\p) = \sdim(A/\p);$ it's important to note that $\mathcal{D}(M)$ could also be defined as the set of all primes in $V(M)$ with $\dim_A(M) = \dim_A(A/\p)$ since the minimal elements in the sets $*V(M)$ and $V(M)$ are exactly the same.   If a homogeneous element $x$ is not in any prime of $\mathcal{D}(M)$, then $\sdim_A(M/xM) < \sdim_A(M);$ this is true for exactly the same reason as in the ungraded case: $*V(M/xM) = *V((x) + Ann_A(M))$.   

Finally, one proceeds to the proof by first arguing that $\sdim_A(M) \leq *d(M)$, then $*d(M) \leq *s(M)$, and lastly, $*s(M) \leq \sdim_A(M)$.

For the first inequality one uses induction on $*d(M)$.  Note that $*d(M) = 0$ means that there is a $q$ such that $*\ell_A(M/\mathcal{N}^iM) = *\ell_A(M/\mathcal{N}^{i+1}M)$ for all $i \geq q$.  But this forces $\mathcal{N}^qM = \mathcal{N}^{q+1}M$ and graded Nakayama's lemma says that $\mathcal{N}^qM = 0$, so that $*V(M)$ has exactly one ideal, $\mathcal{N}$ in it.  By definition, $\sdim_A(M) = 0$.  Supposing that $*d(M) \geq 1$, as in \cite{se}, we reduce to the case $M = A/\p$ for some graded prime ideal $\p$ properly contained in $\mathcal{N}$.  Taking a chain of graded prime ideals $\p \doteq \p_0 \subset \p_1 \subset \cdots \subset \p_n$ in $A$, we may suppose that $n \geq 1$, and thus may choose a homogeneous element $x$ in $\p_1$ that is not in $\p$.  Since $x \notin \p$, but $x \in \p_1$, the chain $\p_1 \subset \cdots \subset \p_n$ corresponds to a chain of primes in $*V(M/xM)$. Since $M = A/\p$, and $x \notin \p$, $_xM = 0$, so that $*p(M/xM, \mathcal{N},n)$ has degree strictly less than $*d(M)$, and by induction, $\sdim_A(M/xM) \leq *d(M/xM)$. Thus, $n-1 \leq \sdim_A(M/xM)  \leq *d(M)-1$ and $n \leq *d(M)$.  This forces $\sdim_A(M) \leq *d(M)$.

For the second inequality, if $x_1, \ldots, x_k$ is a list of homogeneous elements of $\mathcal{N}$ that generate a GIOD $\I$ for $M$, we must have that $*V(\I+Ann_A(M))$ contains only $\mathcal{N}$, so that $*p(M, \I, n)$ and $*p(M, \mathcal{N}, n)$ have the same degree $*d(M)$. But, Lemma \ref{lemma Samuel props} says that $*p(M, \I, n)$ has degree less than or equal to $k$.  Thus, $*d(M) \leq *s(M)$.

For the third inequality, use induction on $\sdim_A(M)$, which we may assume to be at least 1, since  $\sdim_A(M) = 0$ if and only if $\mathcal{N}$ is the only prime in $*V(M)$, so that $M$ has finite *length and $*s(M)= 0$ by definition.  If $\sdim_A(M) \geq 1$, none of the primes in $\mathcal{D}(M)$ are *maximal, so there is a homogeneous element $x \in \mathcal{N}$ such that $x$ is not in any of the primes in $\mathcal{D}(M)$.  We've noted above  that $*s(M) \leq *s(M/xM)+1$ and  $\sdim_A(M) \geq \sdim_A(M/xM) +1. $  These inequalities plus the induction hypotheses give us the result.
\end{proof}

If $R$ is a positively graded ring with $R_0 =k$ a field, then $(R,R_+)$ is a *local ring, so we may apply the fundamental theorem for *local dimension. On the other hand, recall Smoke's dimension theorem (theorem \ref{theorem Smoke dimension}). For any $M \in \grmod(R)$ the hypotheses for Smoke's dimension theorem are satisfied, and we may therefore combine the two dimension theorems.

\begin{corollary}  \label{corollary fun thm dimension theory pos graded}If $R$ is a positively graded Noetherian ring with $R_0$ a field and $M \in \grmod(R)$, then 
$$*\dim_R(M) = *d(M) = *s(M)= s_1(M) = d_1(M) = \dim_R(M).$$
\end{corollary}

Going back to the definition of a GSOP for the $A$-module $M$, as a corollary to Theorem \ref{theorem fun thm graded dimension theory}. we have
\begin{corollary}  If $(A, \mathcal{N})$ is a *local graded Noetherian ring and $M \in \grmod(A)$, then a GSOP exists for $M$, and the length of every GSOP is equal to 
$*\dim_A(M) = *d(M) = *s(M).$  Moreover, if $A-\mathcal{N}$ has a homogeneous element $T$ of degree 1, necessary invertible, and $d(M_0)$ is the degree of the ordinary Samuel polynomial $p(M_0, \mathcal{N}_0, n)$, then $d(M_0) = *d(M)$, and if $x_1, \ldots, x_D$ is a GSOP for $M$, where $D = \sdim_A(M) = *d(M) = d(M_0)$, then $x_1T^{-e_1}, \ldots, x_DT^{-e_D}$ is an ordinary system of parameters for $M_0$ as an $A_0$-module, if $e_i = \deg(x_i)$.                                                
\end{corollary} 
\begin{proof}  The first statement is clear using \ref{theorem fun thm graded dimension theory}; for the second use Theorem  \ref{theorem star e equals e, element of degree 1}  to see that $*d(M)= d(M_0)$;  if $x_1, \ldots, x_D$ generate a GIOD $\I$ for $M$, then $\I_0$ is generated by $x_1T^{-e_1}, \ldots, x_DT^{-e_D}$. Therefore, the ungraded dimension theorem  ensures that $D = d(M_0) = \dim_{A_0}(M_0)$, so \newline $x_1T^{-e_1}, \ldots, x_DT^{-e_D}$  is an ordinary system of parameters for $M_0$.
\end{proof} 
We also have a corollary to Corollary \ref{corollary sum decomp for multiplicity}; here $\mathcal{D}(M)$ is defined as the set of minimal primes of maximal dimension (as in the proof of Theorem \ref{theorem fun thm graded dimension theory})

\begin{corollary} \label{corollary *Samuel sum decomposition} Suppose that $(A, \mathcal{N})$ is a *local graded Noetherian ring, $M \in \grmod(A)$, $\I$ is a GIOD for $M$ and $M^{\bullet}$ is a graded filtration of $M$ such that $0 = M^0 \subset M^1 \subset \cdots M^{N-1} \subset M^N = M,$ and, for each $N \geq i \geq 1$, there are graded prime ideals $\p_i$ in $A$, integers $d_i$ and graded isomorphisms of $A$-modules $(A/\p_i )(d_i) \cong M^{i}/M^{i-1}$.   Then, if $D \doteq \sdim_A(M)$,

$$*e(M, \I, D) = \sum_{\p \in \mathcal{D}(M)} *\ell_{A_{[\p]}}(M_{[\p]}) (*e(A/\p, \I, D)).$$ 
\end{corollary}
\begin{proof}  Lemma \ref{corollary the filtration graded} tells us  that, for every minimal prime $\p$ for $M$, there is at least one subquotient of the filtration isomorphic to the graded $A$-module $A/\p$, possibly shifted.  Therefore, adding the *local hypothesis,
\begin{itemize}
\item $\mathcal{D}(M^{\bullet}) = \mathcal{D}(M)$ since we now know that, for every shift $d$, the degree of  $*p((A/\p)(d), \mathcal{I}, n)= *p(A/\p,\mathcal{I},n)$ is independent of the choice of $\mathcal{I}$ and is equal to $\sdim_A(A/\p)$.  Theorem \ref{theorem fun thm graded dimension theory}  also tells us that  the $D$ in this corollary is exactly the same $D$ as in Corollary \ref{corollary sum decomp for multiplicity}.
\item  Moreover, for every prime $\p$ in $\mathcal{D}(M)$, $A/\p$, possibly shifted, occurs exactly $*\ell_{A_{[\p]}}(M_{[\p]}) = \ell_{A_{\p}}(M_{\p})$  times (using Theorem \ref{theorem *composition series exists for Mgrp when p is minimal} and Corollary \ref{corollary sum decomp for multiplicity}) as a subquotient of the filtration $M^{\bullet}$, so that $n_{\p}(M^{\bullet}) = *\ell_{A_{[\p]}}(M_{[\p]}) = \ell_{A_{\p}}(M_{\p})$.
\end{itemize}

\end{proof}

\section{Koszul complexes in grmod(A) and *Koszul multiplicities}

In this section, $A$ is any graded Noetherian ring and $M \in \grmod(A)$.  We again follow the discussion in \cite{se}.

The definition of a complex of modules in $\grmod(A)$ is as usual: this is a sequence $(\mathbf{M}, \partial) $
$$\cdots \rightarrow M_j \stackrel{\partial_j}{\rightarrow} M_{j-1} \stackrel{\partial_{j-1}}{\rightarrow} \cdots \rightarrow \stackrel{\partial_1}{\rightarrow} M_0$$
of objects and morphisms in $\grmod(A)$, such that $\partial \partial = 0$ everywhere.  The sequence of morphisms $\partial$ above is called the differential for the complex.

The subscripts $j$ seem assigned ambiguously, but here's what we mean: If $(\mathbf{M},\partial)$ is a complex in $\grmod(A)$ as above, then the set of elements of $M_j$ of degree $i$ is equal to
$$(M_j)_i \doteq M_{j,i}.$$

In other words, when speaking of a complex in $\grmod(A)$, a single integer subscript denotes the sequential index of the complex, and a doubly-indexed subscript is read as ``first index is the complex index, second is the graded-module index". We will often suppress the internal gradings, so if there is just one subscript, it refers to the ``complex index".  Hopefully this won't be too confusing.

To further set notation, we will regard any $M$ in $\grmod(A)$ as a ``complex concentrated in degree 0"--this is the complex with all differentials equal to zero, with $M_i = 0$, if the ``complex-index" $i \neq 0$, and $M_0 = M$, for ``complex-index " $0$.

The homology groups of a complex $(\mathbf{M}, \partial)$ are defined as $``\ker \partial/ \mbox{im} \;\partial"$ of course, and are also in $\grmod(A)$:

$$H_j(\mathbf{M})_i = \ker(\partial:M_{j,i} \rightarrow M_{j-1,i})/\mbox{im}(\partial:M_{j+1,i} \rightarrow M_{j,i}).$$

Morphisms of graded complexes and short exact sequences of graded complexes are defined in the usual way.

A short exact sequence of graded complexes in $\grmod(A)$ gives rise to a long exact sequence on homology:  if $0 \rightarrow \mathbf{A} \xra{\alpha} \mathbf{B} \xra{\beta} \mathbf{C} \rightarrow 0$ is a short exact sequence of graded complexes in $\grmod(A)$, there exists a graded morphism $\mathbf{\omega}$ of complex degree $-1$  such that the sequence $$\cdots \xra{\omega_{j+1}} H_j(\mathbf{A}) \xra{\alpha_*} H_j(\mathbf{B}) \xra{\beta_*} H_j(\mathbf{C}) \xra{\omega_j} H_{j-1}(\mathbf{A}) \xra{\alpha_*} \cdots $$
is an exact sequence in $\grmod(A)$.

\begin{definition}

The \textbf{*Euler characteristic} of a complex $(\mathbf{M}, \partial)$ in $\grmod(A)$ is defined when $(\mathbf{M}, \partial)$ is  such that each $A$-module $M_i$  has $*\ell_A(M_i) < \infty$ and for all but finitely many $i$, $*\ell_A(M_i) = 0$ . Given these conditions, the following sum is well defined: $$*\chi(\mathbf{M}) \doteq \Sigma_i (-1)^i *\ell_A(M_i). $$
\end{definition}

Since $*\ell$  sums over short exact sequences, we get the following lemma.
\begin{lemma} \label{lemma euler splits over ses} Let $0 \rightarrow \mathbf{A} \rightarrow \mathbf{B} \rightarrow \mathbf{C} \rightarrow 0$ be a short exact sequence of graded complexes in $\grmod(A)$.  Then, the *Euler characteristic of $\mathbf{B}$ is defined if and only if the *Euler characteristics of $\mathbf{A}$ and $\mathbf{C}$ are, and  $*\chi(\mathbf{B}) = *\chi(\mathbf{A}) + *\chi(\mathbf{C})$.
\end{lemma}

If the two conditions for a well-defined *Euler characteristic of a complex are not met, there may be a way to salvage the situation by passing to homology. 

\begin{definition}
Let $(\mathbf{M}, \partial)$ be a complex in $\grmod(A)$ such that for every $i$, $H_i(\mathbf{M})$ has finite *length over $A$, and for $i >> 0$ $*\ell(H_i(\mathbf{M})) = 0$.  We define the *Euler characteristic of the homology to be $*\chi(\mathbf{H(M)}) \doteq \Sigma_i (-1)^i *\ell_A(H_i (\mathbf{M}))$.
\end{definition}

With the proof exactly analogous to that in the ungraded case, we have

\begin{theorem}
\label{theorem Euler of homology equals Euler of complex}
When the *Euler characteristic $*\chi(\mathbf{M})$ is defined, then $*\chi(\mathbf{H}(\mathbf{M}))$ is also defined, and  we have that $*\chi(\mathbf{M}) = *\chi(\mathbf{H(M)})$.
\end{theorem}

Note however that the converse is not necessarily true; i.e. $*\chi$ of the homology may be defined but $*\chi$ of the complex not.

Using the additivity of $*\ell$, and the long exact sequence on homology, if $\mathbf{A} \rightarrowtail \mathbf{B} \twoheadrightarrow \mathbf{C}$ is a short exact sequence of graded complexes in $\grmod(A)$ such that the *Euler characteristic of the homology of each complex is defined, then, $$*\chi(\mathbf{H(B)}) = *\chi(\mathbf{H(A)}) + *\chi(\mathbf{H(C)}).$$

If $A$ is a graded ring, we may do homological algebra in $\grmod(A)$ quite analogously to how it's done in the  ungraded case.  In particular, graded $A$-modules $Tor_i^A(M,N) \in \grmod(A)$ for every $i \geq 0$, and $M,N \in \grmod(A)$ may be defined mimicking the constructions and definitions in the ungraded category: beginning with  the graded tensor product $M \otimes_A N$. (For the definition of the graded tensor product of graded modules over a graded ring, see \cite{GRO}.) The tensor product $M \otimes_A N$ has a natural grading on it:  if $m \in M_i$ and $n \in N_j$ are homogeneous elements, then $\deg(m\otimes n) = i + j$. Then, one proceeds to talk about projective resolutions, and arrives at the definition of $Tor_i^A(M,N) \in \grmod(A)$.  We do not give further details here.

\subsection{The Koszul complex}

Standard properties of the Koszul complex in the ungraded case may be found in \cite{se}, Chapter IV.  We use Serre's notation: if $\bar{x} \doteq x_1, \ldots, x_u$ is a sequence of elements in $A$, then the Koszul complex is denoted by $\mathbf{K}(\bar{x},A)$.  

If we pass to the graded category, with $A$ a graded ring, and choose a sequence $\bar{x} \doteq x_1, \ldots, x_u$ of homogeneous elements of $A$, the definition of the graded Koszul complex is briefly summarized as follows.   Recall   that  the tensor product of graded complexes $\mathbf{C} \otimes_A \mathbf{D}$ is defined exactly analogously to the ungraded case, and is again a  graded complex; keep in mind in particular the definition of the differential of a tensor product of complexes:  if $c \in C_i$ and $d \in D_j$, then $\partial_{C \otimes_A D}(c \otimes d) = \partial_C(c) \otimes d + (-1)^i c \otimes \partial_D(c).$  Starting with the case $u=1$, $\mathbf{K}(x_1,A)$ is the two-term complex in $\grmod(A)$
$$K_1(x_1, A) = A(-d) \stackrel{\cdot x_1}{\rightarrow} K_0(x_1,A) = A,$$ where $d$ is the degree of $x_1$.  Then, if $\bar{x} = x_1, \ldots, x_u$ is a sequence of homogeneous elements in $A$,
$$\mathbf{K}(\bar{x},A) \doteq \mathbf{K}(x_1,A) \otimes_A \cdots \otimes_A \mathbf{K}(x_u,A).$$ 

If $M \in \grmod(A)$, the Koszul complex associated to the graded $A$-module $M$ and $\bar{x}$ is:
$$\mathbf{K}^A(\bar{x},M) \doteq \mathbf{K}(\bar{x},A) \otimes_A M.$$  If we're always regarding a graded abelian group $M$ as an $A$-module, for a fixed graded ring $A$, we will often delete the superscript $A$.

Setting notation,  $K_0(x_i, A)$ is identified with $A$ as a free, graded $A$-module (in other words, the free generator lies in degree zero, and is identified with $1\in A_0$).  For $K_1$, choose $e_{x_i}$ of  $\deg(x_i)$ and identify $K_1(x_i,A)$ with the free graded $A$-module generated by the homogeneous element $e_{x_i}$.  Then, $K_p(\bar{x}, A)$ is the free graded $A$-module isomorphic to the free graded $A$-module generated by the homogeneous elements $e_{x_{i_1}} \otimes \cdots \otimes e_{x_{i_p}}$ of degree $\deg(x_{i_1}) + \cdots +\deg(x_{i_p})$ where $i_1 < \cdots < i_p$, so is isomorphic to the graded exterior product 
$$\Lambda^p(A(-\deg(x_1)) \oplus \cdots \oplus A(-\deg(x_u))).$$ In addition, both the $p$th part of the Koszul complex $\mathbf{K}^A(\bar{x}, M)$, and its differential have exactly the same form as described in \cite{se}, IV.A.2 in the ungraded case.  A particular consequence  is that, as a graded $A$-module, $K_p^A(\bar{x}, M)$ is a direct sum of ${u \choose p}$ copies of $M$, each shifted:  the copy associated to the multi-index $i_1 <  \cdots i_p$ looks like $M(-(\deg(x_{i_1}) + \cdots +\deg(x_{i_p})))$;  if $\I$ is the graded ideal of $A$ generated by $x_1, \ldots, x_u$, and $k \geq 0$, then $K_p^A(\bar{x}, M)/ \I^k K_p^A(\bar{x}, M)$ is, as a graded $A$-module, isomorphic to ${u \choose p}$ copies of $M/\I^kM$, each shifted as described above.

The $p$th homology group of the graded Koszul complex $\mathbf{K}^A(\bar{x},M)$ is denoted by $H_{p}(\bar{x}, M),$ or $H^A_p(\bar{x}, M)$ if we need to emphasize the role of $A$. These homology groups are also graded $A$-modules.

\begin{definition}     Suppose that $x_1, \ldots , x_u$ is a sequence of nonzero nonunit homogeneous elements in $A$ and $M \in \grmod(A)$. This sequence is a $M$-sequence if and only if $x_1$ is not a zero-divisor on $M$, and
for each $i>1$, $x_i$ is not a zero-divisor on $M/(x_1, \ldots, x_{i-1})M$. 
\end{definition}

The following may all be proved as in the ungraded case (see \cite{se}, Chapter IV): 
\begin{proposition}  Let $A$ be a graded ring and $M \in \grmod(A)$.
If $\bar{x}$ is a $M$-sequence, then the Koszul complex $\mathbf{K}^A(\bar{x},M)$ is acyclic.  As in the ungraded case, $H_0^A(\bar{x},M) = M/(x_1,\ldots,x_u)M$. 
\end{proposition}

Conversely, in the *local Noetherian case one has
\begin{proposition}  If $(A, \mathcal{N})$ is a *local Noetherian ring, and $M \in \grmod(A)$, then the following are equivalent, for a sequence of homogeneous elements $\bar{x} \doteq x_1, \ldots, x_u$ of $\mathcal{N}$:
\begin{itemize}
\item $H_p^A(\bar{x}, M) = 0,$ for $p \geq 1$.
\item $ H_1^A(\bar{x}, M) = 0.$
\item  $\bar{x}$ is an $M$-sequence in $A$.
\end{itemize}
\end{proposition}

The proofs of the above Propositions are exactly analogous as that of  IV.A.2, Propositions 2, 3 in \cite{se}, replacing any use of Nakayama's lemma with the graded Nakayama's lemma (Lemma \ref{lemma graded nakayama}); similarly, IV.A.2, Corollary 2 yields, in the graded case,
\begin{corollary}  If $(A, \mathcal{N})$ is a *local Noetherian ring, $M \in \grmod(A)$, and $\bar{x} = x_1, \ldots, x_u$ are homogeneous elements of $\mathcal{N}$ that form an $A$-sequence for $A$, then there is a natural isomorphism of graded $A$-modules
$$\psi: H_i^A(\bar{x}, M) \rightarrow Tor_i^A(A/(\bar{x}), M).$$
\end{corollary}

Finally, IV.A.2, Proposition 4, has the analogous
\begin{proposition}  Suppose that $(A, \mathcal{N})$ is a *local graded Noetherian ring and $M \in \grmod(A)$.  If $x_1, \ldots, x_u$ are homogeneous elements of $\mathcal{N}$, then $(\bar{x}) + Ann_A(M) \subseteq Ann_A(H_i^A(\bar{x}, M)).$
\end{proposition}

As a corollary, 
\begin{proposition}  \label{corollary homology fdvs} Suppose that $(A, \mathcal{N})$ is a *local graded Noetherian ring , and $M \in  \grmod(A)$.   Let $\I$ be a GIOD for $M$, generated by the homogeneous sequence $\bar{x} = x_1, \ldots , x_u \in \mathcal{N}$. Then, $H_j^A(\bar{x}, M) $ has finite *length over $A$ for every $j \geq 0$.
\end{proposition}
\begin{proof}  Since $\I + Ann_A(M) \subseteq Ann_A(H_j^A(\bar{x}, M))$, and $\{ \mathcal{N}\} = *V(\I + Ann_A(M))$, if \newline $H_j^A(\bar{x}, M) \neq 0, \{ \mathcal{N} \} = *V(Ann_A(H_j^A(\bar{x}, M))).$

\end{proof}

Thus, in the *local case the *Euler characteristic of the homology of the graded Koszul complex is well defined:

\begin{definition} Suppose that $(A, \mathcal{N})$ is a *local, Noetherian graded ring and $M \in \grmod(A)$.  
Let $\I$ be a GIOD for $M \in \grmod(A)$  generated by a homogeneous sequence $\bar{x}=x_1, \ldots, x_u$. We define the *\textbf{Koszul multiplicity} $*\chi^A(\bar{x},M)$ to be the *Euler characteristic of the homology of the graded Koszul complex:   $$*\chi^A(\bar{x},M) = \sum_{i=1}^u(-1)^i *\ell_A(H_i^A(\bar{x},M)).$$
\end{definition}

\subsection{Equality of *Samuel and *Koszul multiplicities}
As in the ungraded case \cite{se}, IV.A.3, the *Koszul multiplicity is equal to a certain *Samuel multiplicity.  This section concludes our account of the theory of multiplicities adapted to the $\mathbb{Z}$-graded category. 

Let $(A, \mathcal{N})$ be  a *local, Noetherian graded ring, $M \in \grmod{A}$ and $\bar{x} =x_1, \ldots, x_u$ a sequence of homogeneous elements contained in $\mathcal{N}$. If $\I$ is the graded ideal of $A$ generated by $\bar{x}$, suppose also that $\I$ is a GIOD for $M$.  

One then filters the graded Koszul complex, yielding  graded complexes $\mathcal{F}^i\mathbf{K}$ for every $i$, with   $\mathcal{F}^iK_p \doteq \I^{i-p}K_p$ for every $p$ (we've dropped the arguments $\bar{x}, M$ for expediency).  Notice we have three indices now:  the filtration index, the complex index and the internal gradings of the various $A$-modules involved.  We are suppressing the internal grading. This filtration defines the  associated graded complex  $gr(\mathbf{K}) \doteq \oplus_i \mathcal{F}^i\mathbf{K}/ \mathcal{F}^{i+1} \mathbf{K}$.

If $gr(A)$ is the associated bigraded ring to the $\I$-adic filtration, then denote the images of $x_1, \ldots, x_u$ in $gr(A)_{1,*}$ by $\xi_1, \ldots, \xi_u$.  Let $gr(M)$ be the bigraded $gr(A)$-module associated with the $\I$-adic filtration of $M$. Then, there is an isomorphism of graded objects $gr(\mathbf{K}) \cong \mathbf{K}(\bar{\xi}, gr(M))$.  Moreover, one argues that the homology modules $H_p(\bar{\xi},gr(M))$ have finite *length over $gr(A)$,  for all $p$, since $A/(\I + Ann_A(M))$ is *Artinian.  This in turn, enables one to argue that there exists an $m \geq u$ such that the graded homology groups of the complex $\mathcal{F}^i\mathbf{K}/ \mathcal{F}^{i+1} \mathbf{K}$ all vanish for $i >m$, and so one sees that the graded homology groups the complex $\mathcal{F}^i\mathbf{K}$ all vanish if $i >m$.

Continuing as in \cite{se}, IV.A.3, (which is really a spectral sequence argument), this means there is an $m$ such that if $i >m$, then $H_p(\mathbf{K}) \cong H_p(\mathbf{K}/\mathcal{F}^i \mathbf{K})$ for $i >m$ and for all $p$.

Using the fact that *Euler characteristics don't change when passing to homology, $*\chi(\bar{x},M) = \sum_p (-1)^p *\ell(H_p(\mathbf{K}/\mathcal{F}^i(\mathbf{K})) = *\chi(\mathbf{K}/\mathcal{F}^i\mathbf{K})$, for  $i >m$.  As noted in the previous section,  $(\mathbf{K}/\mathcal{F}^i\mathbf{K})_p$ is isomorphic as a graded $A$-module to a direct sum of ${u \choose p}$ copies of $M/\I^{i-p}M$, shifted appropriately,  and since the *length of a shifted $A$-module $M(d)$ is the same  as the *length of $M$, the remainder of the proof is argued exactly as in \cite{se}, IV.A.3, with length replaced by *length, $p$ (a Samuel polynomial)  replaced by $*p$ and $e$ replaced by $*e$.

Thus, we have
\begin{theorem} \label{theorem koszul equals samuel}
Let $(A, \mathcal{N})$ be a *local Noetherian ring.  Let $x_1, \ldots, x_u \in \mathcal{N}$ be homogeneous elements generating a graded ideal of definition $\I$ for $M \in \grmod(A)$. Then, $$*\chi^A(\bar{x},M) = *e(M,\I,u),$$  so $*\chi^A(\bar{x}, M)$ is a strictly positive integer if $\sdim_A(M) = u$, and $*\chi^A(\bar{x}, M) = 0$ if $u > \sdim_A(M)$.
\end{theorem}

\section{Multiplicities and Degree for Positively Graded Rings}
\label{section Multiplicities and Degree for Positively Graded Rings}
In this section, we specialize to the case of a positively graded Noetherian ring $R$ with $R_0$ a field $k$; all graded modules are in $\grmod(R).$  Set $\m = R_+$ and note that $(R, \m)$ is then a *local graded Noetherian ring.  We do not want to make the assumption that $R$ is generated by elements in degree 1. 

In this chapter we may use the * notation even though we could just as well omit the * (e.g. If $M \in \grmod(R)$, then $*\dim_R(M)=\dim_R(M)$.) This is done to emphasize the fact that all computations may be done in the graded category using the theory developed in the previous two chapters (which is often simpler than the ungraded theory.)  

We introduce the \textit{degree} of a graded module, show how it relates to *multiplicity (Theorem \ref{theorem main theorem}), and give a sum decomposition of degree by a certain set of minimal primes (Theorem \ref{theorem sum formula for degree and multiplicity}.)  
 
Since $R_0=k$,  $\ell_k(M_i) = \vdim_k(M_i)$ for every $i$, so the Poincar\'{e} series for $M$ is equal to 
$$P_M(t) = \sum_{i  \in \mathbb{Z}} \vdim_k(M_i) t^i.$$ Furthermore, this Laurent series has a pole at $t=1$ using the Hilbert-Serre theorem, and the order of the pole $d_1(M)$ at $t=1$ is, by Smoke's dimension theorem, is exactly $\sdim_R(M)$.

This leads to the definition of $\deg_R(M)$:

\begin{definition}  If $R$ is a positively graded Noetherian ring with $R_0=k$ a field, $M\in \grmod(R), \newline M \neq 0$ and $D(M) = \sdim_R(M)$, then
$$\deg_R(M) \doteq \lim_{t \rightarrow 1} (1-t)^{D(M)}P_M(t)$$ is a well-defined, strictly positive, rational number.  For convenience, define $\deg_R(0) = 0.$
\end{definition}

Often we delete the subscript $R$ and just write $\deg(M)$.  We use the (somewhat ambiguous) name of ``degree" for this rational number in deference to the nomenclature already used in \cite{Bensonea}.  For equivariant cohomology, this ``degree" was first studied by Maiorana \cite{ma}.

\section{Multiplicities and Euler-Poincar\'{e} series}

If $X \in \grmod(R)$ has finite *length as an $R$-module, since each $R_i$ is finite-dimensional as a vector space over $k$, we may use Lemmas \ref{lemma pos graded length equals *length} and \ref{lemma *length of components} to conclude that  $\ell_R(X) = *\ell_R(X) = \vdim_k(X)$, where $\vdim_k(X)$ is the total dimension $\sum_j \vdim_k(X_j)$ of the graded $k$-vector space $X$.  We may then prove:

\begin{lemma} Suppose $R$ is a positively graded Noetherian ring with $R_0 = k$, a field,  and $X \in \grmod(R)$ is such that $*\ell_R(X) < \infty$.  If $B$ is a graded subring of $R$, Noetherian or not, with $B_0 = k = R_0$, then $X \in \grmod(B)$,  $*\ell_B(X) < \infty$ and  $*\ell_B(X) = *\ell_R(X) =  \ell_R(X) = \vdim_k(X) < \infty$.
\end{lemma}
\begin{proof}  For, using Lemma \ref{lemma *length of components}  applied to $R$, each $X_i$ is finite-dimensional over $k$,  and there are integers $t_0$ and $J$ such that $t_0 \leq J$ with $X= \oplus_{j=t_0}^J X_j$.  Also, $*\ell_R(X) = \vdim_k(X)$.  However, since $k \subseteq B \subseteq R$, $X$ is a finitely generated  $B$-module.   Whether $B$ is Noetherian or not, since $B_j \subseteq R_j$ for every $j$, and $R_j$ is finite-dimensional over $k$, so is $B_j$.  Thus, using Lemma \ref{lemma *length of components}  applied to $B$, $*\ell_B(X) = \vdim_k(X)$ as well. \end{proof}

If $M \in \grmod(R)$, then $\m$ is a GIOD for $M$, and we may then calculate a *Samuel polynomial $*p_R(M, \m, n)$ for $M$; Theorem  \ref{theorem star e equals e, pos graded} says that this is the ordinary Samuel polynomial $p_R(M, \m, n)$; Corollary \ref{corollary fun thm dimension theory pos graded} says that the degree of this polynomial is 
$$D(M) \doteq *d(M) = *s(M) = s_1(M) = d_1(M) = \sdim_R(M) = \dim_R(M).$$

Now, suppose that $\bar{x} = x_1, \ldots, x_{D(M)}$ is a GSOP for the $R$-module $M$.  If $\I$ is the graded ideal in $R$ generated by $\bar{x}$, $\I$ is a GIOD for $M$.  We can change rings to $B \doteq k\langle x_1, \ldots, x_{D(M)} \rangle$, note that this is a graded polynomial ring over $k$ in the indicated variables (Proposition \ref{proposition algebraic independence}).  The ideal $\hat{\I}$ generated by $\bar{x}$ in $B$ is also a GIOD in $B$ since clearly $\hat{\I}^n M = \I^nM$ for every $n$.  Therefore Theorem \ref{theorem star e equals e, pos graded} and the previous lemma guarantee that, for every $n$, the polynomials below are all equal, as indicated:
$$*p_R(M, \I, n) = p_R(M, \I, n) = p_B(M, \hat{\I}, n) = *p_B(M, \hat{\I}, n);$$ in particular, they all have the same degree $D(M)$, and the following positive integers are also all equal:
$$*e_R(M, \I, D(M)) = e_R(M, \I, D(M)) = e_B(M, \hat{\I}, D(M)) = *e_B(M, \hat{\I}, D(M)).$$

\subsection{Euler-Poincar\'{e} series}

The following lemma is found in Avramov and Buchweitz \cite{avbu}; \cite{sm} contains a similar result.

\begin{lemma} (Lemma 7  of \cite{avbu}) \label{corollary Euler-Poincare series props}If $M,N \in \grmod(R)$, then
\begin{enumerate}
\item[a.] For each $i$, the graded $R$-module $Tor_i^R(M,N)$ has finite dimensional (over $k = R_0$) homogenous components $Tor_i^R(M,N)_j$, for every $j$; also, for every $i$,  $Tor_i^R(M,N)_j = 0$ for $j <<0$. Thus one may form the Laurent series  $$P_{Tor_i^R(M,N)}(t) \doteq \sum_{j \in \mathbb{Z}} \vdim_k(Tor_i^R(M,N))_jt^j.$$
\item[b.] Furthermore,the alternating sum
$$\chi_R(M,N)(t) \doteq \sum_{i \geq 0} (-1)^i P_{Tor_i^R(M,N)}(t),$$ which is by definition the Euler-Poincar\'{e} series of $M,N$, is a well-defined Laurent series with integer coefficients and
$$P_R(t)\chi_R(M,N)(t) = P_M(t)P_N(t).$$
\end{enumerate}
\end{lemma}

\label{section Multiplicities and the Euler-Poincare Series}
If a GSOP $\bar{x}$ is given for $M \in \grmod(R)$,  $B \doteq k \langle \bar{x} \rangle \subseteq R$  is then a graded polynomial ring over $k$ (Proposition \ref{proposition algebraic independence}), and  $M \in \grmod(B)$, using Lemma \ref{lemma finite generation gsop}.  Whether we consider $M \in \grmod(R)$, or $M \in \grmod(B)$, the Poincar\`{e} series of $M$ does not change.

Hilbert's Syzygy Theorem tells us that the graded Koszul complex $\mathbf{K}^{k\langle \bar{x} \rangle}(\bar{x}, k)$ is acyclic, thus is a free, finite graded resolution of $k$ as a graded $k\langle \bar{x} \rangle$-module.  In particular, we may tensor this resolution with $M$  and use it to compute
$Tor_i^{k\langle \bar{x} \rangle}(k,M) = Tor_i^{k\langle \bar{x} \rangle}(M,k),$ showing that
$$Tor_i^{k\langle \bar{x} \rangle}(M,k) = H^{k\langle \bar{x} \rangle}_i(\bar{x}, M).$$

\begin{lemma} \label{star chi equals chi} Let $\bar{x} = x_1, \ldots, x_{D(M)}$ be a GSOP for $M \in \grmod(R)$, and let $\I$ be the graded ideal in $R$ generated by $\bar{x}$.
For every $i$, $ P_{Tor_i^{k\langle \bar{x} \rangle}(M,k)}(t) \in \mathbb{Z}[t, t^{-1}],$ and therefore
$\chi_{k\langle \bar{x} \rangle}(M,k)(t) \in \mathbb{Z}[t,t^{-1}].$  Furthermore, 
$$\chi_{k\langle \bar{x} \rangle}(M,k)(t) = \sum_{j=0}^{D(M)} (-1)^jP_{H_j^{k\langle \bar{x} \rangle}(\bar{x},M)}(t),$$ 
and  evaluating this Laurent polynomial at $t=1$, we compute
$$\chi_{k\langle \bar{x} \rangle}(M,k)(1)  = *\chi^{k\langle \bar{x} \rangle}(\bar{x}, M) =  *e_R(\I, M, D(M))= *\chi^R(\bar{x}, M),$$  where $D(M) = \sdim_R(M).$
\end{lemma}

\begin{proof}   Lemma \ref{corollary Euler-Poincare series props} shows the first part of the statement, and since the resolution $\mathbf{K}^{k\langle \bar{x} \rangle}(\bar{x}, k)$ is in any case zero for complex degree larger than $D(M)$,
$Tor_i^{k\langle \bar{x} \rangle}(M,k)$ is also zero for $i > D(M)$, so
$$\chi_{k \langle \bar{x} \rangle}(M,k)(t) \doteq \sum_{j \geq 0} (-1)^j P_{Tor_j^{k \langle \bar{x} \rangle}(M,k)}(t) =  \sum_{j=0}^{D(M)} (-1)^jP_{H_j^{k\langle \bar{x} \rangle}(\bar{x},M)}(t),$$ being a finite sum of Laurent polynomials, is a Laurent polynomial.

Setting $B \doteq k\langle \bar{x} \rangle$ yields,
$$*\ell_B(H_i^B(\bar{x}, M)) = \ell_B(H_i^B(\bar{x},M)) = \vdim_k(H_i^B(\bar{x}, M)),$$ so
$$*\chi^B(\bar{x}, M) = \sum_{j=0}^{D(M)} (-1)^j \ell_B(H_j^B(\bar{x}, M))= \sum_{j=0}^{D(M)} (-1)^j \vdim_k(H_j^B(\bar{x}, M)) \doteq \chi_B(M,k)(1).$$

As noted at the beginning of this section,  if $\hat{\I}$ is the ideal generated by $\bar{x}$ in $B$, then \newline $*e_B(M, \hat{\I}, D(M)) = *e_R(M, \I, D(M))$ and Theorem \ref{theorem koszul equals samuel} says that $$*\chi^B(\bar{x}, M) = *e_B(M, \hat{\I}, D(M))  = *e_R(M, \I, D(M)) = *\chi^R(\bar{x}, M) .$$  

\end{proof}

\subsection{Degree of a Graded Module in $\grmod(R)$}

Given $M \in \grmod(R)$,  $\deg(M) >0$, if $M \neq 0$,  
we can read off the degree of a module directly from the Poincare series if we expand it as a Laurent series about $t=1$: $$P_R(t) = \frac{\deg(M)}{(1-t)^{D(M)}} + \text{"higher order terms"}.$$
 
\begin{lemma} Suppose that $0 \rightarrow N \rightarrow M \rightarrow P \rightarrow 0$ is an exact sequence in $\grmod(R)$. Then,
\begin{itemize}
\item $D(M) = \max \{D(N), D(P) \}. $
\item If $D(N) < D(M),$ then $\deg(M) = \deg(P).$
\item If $D(P) <D(M),$ then $\deg(M)=\deg(N).$
\item If $D(P)=D(N)=D(M)$, then $\deg(M)=\deg(N) + \deg(P).$
\item $\deg(M(d)) = \deg(M)$, for every integer $d$.
\end{itemize}
\end{lemma}

This immediately yields, as in \cite{Bensonea}:

\begin{theorem}  \label{theorem sum formula for degree and multiplicity}
Let $M \in \grmod(R)$,  and $\mathcal{D}(M)$ be defined as in  Theorem \ref{theorem fun thm graded dimension theory}:  this is the set of prime ideals $\p$ in $R$, necessarily minimal primes for $M$ and graded, such that $\sdim_R(R/\p) = \sdim_R(M)$.  Then, $$\deg(M) = \sum_{\p \in \mathcal{D}(M)}*\ell_{R_{[\p]}}(M_{[\p]})\cdot \deg(R/\p).$$
\end{theorem}
\begin{proof} Choose a graded filtration $M^{\bullet}$ of $M$ of the form in Lemma \ref{corollary the filtration graded}  We know that if $\p \in \mathcal{D}(M)$, then the graded $R$-module $R/\p$, possibly shifted, occurs exactly $*\ell_{R_{[\p]}}(M_{[\p]}) = \ell_{R_{\p}}(M_{\p})$ times (using Theorem \ref{theorem *composition series exists for Mgrp when p is minimal}) as a subquotient in the filtration.  The lemma above then gives the result.
\end{proof}

We want to compare degree to our previously studied multiplicities.

Letting $\bar{x}$ be  a GSOP for $M \in \grmod(R)$, we've seen that $k\langle \bar{x} \rangle$ is a graded polynomial ring, and one directly calculates that
$$P_{k\langle \bar{x} \rangle}(t) = \frac{1}{\prod_{i=1}^{D(M)} (1-t^{d_i})},$$ where $d_i$ is the degree of the homogeneous element $x_i$.

Now,  using $M \in \grmod(k\langle \bar{x} \rangle)$, and recalling that $P_M(t)$ is the same whether we consider $M \in \grmod(R)$ or  $M \in \grmod(k \langle \bar{x} \rangle)$, Lemma \ref{corollary Euler-Poincare series props} gives that
$$P_k(t)P_M(t) = P_{k \langle \bar{x} \rangle}(t) \chi_{k \langle \bar{x} \rangle}(M,k)(t).$$  Also, $\chi_{k\langle \bar{x} \rangle}(M,k)(t) \in \mathbb{Z}[t,t^{-1}]$.  Since $P_k(t) = 1$, we have

\begin{theorem} If $M \in \grmod(R)$ and $\bar{x}$ is a GSOP for $M$, then
$$P_M(t) = \frac{\chi_{k\langle \bar{x} \rangle}(M,k)(t)}{\prod_{i=1}^{D(M)} (1-t^{d_i})},$$
with $\chi_{k\langle \bar{x} \rangle}(M,k)(t) \in \mathbb{Z}[t,t^{-1}].$
\end{theorem}

Since
$$(1-t)^{D(M)}P_M(t) = \frac{\chi_{k\langle \bar{x} \rangle}(M,k)(t) }{\prod_{i=1}^{D(M)}(1+t + \cdots + t^{d_i-1})},$$  taking the limit as $t$ approaches 1, and using Lemma \ref{star chi equals chi}, yields:

\begin{theorem}  \label{theorem main theorem} If $M \neq 0$ is in $\grmod(R)$, and $x_1, \ldots , x_{D(M)}$ of degrees $d_1, \ldots, d_{D(M)}$ form a GSOP for $M$, generating the graded ideal $\I$ of $R$, then
$$\deg(M) = \frac{*e_R(M, \I, D(M))}{d_1 \cdots d_{D(M)}} = \frac{*\chi^R(\bar{x}, M)}{d_1 \cdots d_{D(M)}}.$$

Thus, the ratio $$\frac{*e_R(M, \I, D(M))}{d_1 \cdots d_{D(M)}}$$ is independent of the choice of system of parameters $x_1, \ldots, x_{D(M)}$ for $M$.

\end{theorem}

Note that we can delete the ``stars"  in the equalities of the above theorem and retain the equalities, using Theorem \ref{theorem star e equals e, pos graded}.  The reader should compare this result to Proposition 5.2.2 of \cite{pr}, which states a similar result for rings with standard grading.

\end{document}